\RequirePackage{fix-cm}
\newif\ifusepgfplot
\usepgfplottrue %comment out to bypass pgfplot and use figures in directory figs/

\documentclass[smallcondensed]{svjour3}     % onecolumn (ditto)
\smartqed  % flush right qed marks, e.g. at end of proof
\usepackage{graphicx}
\usepackage[dvipsnames]{xcolor}
\usepackage{amssymb}
\usepackage{amsmath}
\usepackage{stmaryrd}
\usepackage{placeins}
\usepackage{multirow}
\usepackage{bm}
\usepackage{bbm}
\usepackage{pgfplots}
\pgfplotsset{compat=1.16}
 \usepackage{pgfplotstable}
\usepackage{booktabs}
\usepackage{hyperref}
\usepackage{tikz}
\usepackage{caption}
\usepackage{subcaption}
\usepackage{algorithm}
\usepackage[algo2e]{algorithm2e} 
\usepackage{upgreek}
\usepackage{textgreek}
\usepackage{orcidlink}
\usepackage[misc]{ifsym}

\DeclareMathAlphabet{\mathdutchcal}{U}{dutchcal}{m}{n}

\usetikzlibrary{shapes,arrows,calc,external}

\usepackage{changes}

\DeclareMathOperator{\Real}{Re}
\DeclareMathOperator{\Imag}{Im}

\newtheorem{assumption}{Assumption}

\newcommand{\method}[3]{#1\text{DRK}#2\text{-}#3}
\newcommand{\methodcat}[3]{#1\text{DRKCAT}#2\text{-}#3}

\newcommand{\refEqual}[1]
{
	\mathrel{\overset{\makebox[0pt]{\mbox{\normalfont\tiny\sffamily #1}}}{=}}
}

\newcommand{\logLogSlopeTriangle}[6]
{
    \pgfplotsextra
    {
        \pgfkeysgetvalue{/pgfplots/xmin}{\xmin}
        \pgfkeysgetvalue{/pgfplots/xmax}{\xmax}
        \pgfkeysgetvalue{/pgfplots/ymin}{\ymin}
        \pgfkeysgetvalue{/pgfplots/ymax}{\ymax}

        \pgfmathsetmacro{\xArel}{#1}
        \pgfmathsetmacro{\yArel}{#3}
        \pgfmathsetmacro{\xBrel}{#1-(#5*#2)}
        \pgfmathsetmacro{\yBrel}{\yArel}
        \pgfmathsetmacro{\xCrel}{\xArel}

        \pgfmathsetmacro{\lnxB}{\xmin*(1-(#1-(#5*#2)))+\xmax*(#1-(#5*#2))} % in [xmin,xmax].
        \pgfmathsetmacro{\lnxA}{\xmin*(1-#1)+\xmax*#1} % in [xmin,xmax].
        \pgfmathsetmacro{\lnyA}{\ymin*(1-#3)+\ymax*#3} % in [ymin,ymax].
        \pgfmathsetmacro{\lnyC}{\lnyA+#4*(\lnxA-\lnxB)}
        \pgfmathsetmacro{\yCrel}{\lnyC-\ymin)/(\ymax-\ymin)}

        \coordinate (A) at (rel axis cs:\xArel,\yArel);
        \coordinate (B) at (rel axis cs:\xBrel,\yBrel);
        \coordinate (C) at (rel axis cs:\xCrel,\yCrel);

		\ifnum #5=1
			\def\loc{west};
		\else
			\def\loc{east};
		\fi
			
        \draw[#6]   (A)-- node[pos=0.5,anchor=north] {}
                    (B)-- 
                    (C)-- node[pos=0.5,anchor=\loc] {#4}
                    cycle;
    }
}

\definecolor{grassGreen}{RGB}{26,224,72}
\definecolor{mintGreen}{RGB}{73,186,142}
\definecolor{lightViolet}{RGB}{238,70,238}

\pgfdeclareplotmark{mystar*}{
    \node[star,star point ratio=1.5,inner sep=0pt,minimum size=4.5pt,draw=lightViolet,fill=lightViolet!25!white, fill opacity=0.7] {};
}

\pgfplotscreateplotcyclelist{rainbow}{
	red, every mark/.append style={fill=red!25!white, fill opacity=0.7}, mark=pentagon*\\
	orange, every mark/.append style={fill=orange!25!white,	 fill opacity=0.7}, mark=triangle*\\
	grassGreen, every mark/.append style={fill=grassGreen!25!white, fill opacity=0.7}, mark=square*\\
	mintGreen, every mark/.append style={fill=mintGreen!25!white, fill opacity=0.7}, mark=diamond*\\
	blue, every mark/.append style={fill=blue!25!white, fill opacity=0.7}, mark=*\\
	lightViolet, mark=mystar*\\
}

\allowdisplaybreaks

\usepackage{ifthen}
\newboolean{compilefromscratch}
\setboolean{compilefromscratch}{false}

\begin{document}

\title{Jacobian-free explicit multiderivative Runge-Kutta methods for hyperbolic conservation laws}


\titlerunning{Jacobian-free explicit MDRK methods for hyperbolic conservation laws}        % if too long for running head

\author{Jeremy Chouchoulis  \and Jochen Sch\"utz  \and Jonas Zeifang }

\authorrunning{J. Chouchoulis \and J. Sch\"utz \and J. Zeifang} % if too long for running head

\institute{Jeremy Chouchoulis \orcidlink{0000-0003-3451-2568}, Jochen Schuetz \orcidlink{0000-0002-6355-9130}, Jonas Zeifang \orcidlink{0000-0002-7388-923X} \\
		\email{firstname.lastname@uhasselt.be} \\
% 		Jochen  Sch\"utz\\
% 		\email{jochen.schuetz@uhasselt.be} \\ \\
% 		Jonas Zeifang \\
% 		\email{jonas.zeifang@uhasselt.be} \\
 \at Hasselt University, Faculty of Sciences \& Data Science Institute, Agoralaan Gebouw D, BE-3590 Diepenbeek, Belgium}

\date{Received: date / Accepted: date}

\maketitle

\begin{abstract}
Based on the recent development of Jacobian-free Lax-Wendroff (LW) approaches for solving hyperbolic conservation laws [Zorio, Baeza and Mulet, Journal of Scientific Computing 71:246-273, 2017], [Carrillo and Par\'es, Journal of Scientific Computing 80:1832-1866, 2019], a novel collection of explicit Jacobian-free multistage multiderivative solvers for hyperbolic conservation laws is presented in this work. 
In contrast to Taylor time-integration methods, multiderivative Runge-Kutta (MDRK) techniques achieve higher-order of consistency not only through the excessive addition of higher temporal derivatives, but also through the addition of Runge-Kutta-type stages. This adds more flexibility to the time integration in such a way that more stable and more efficient schemes could be identified. 
The novel method permits {the practical application} of MDRK schemes. In their original form, they are difficult to {utilize} as higher-order flux derivatives have to be computed analytically. Here we overcome this by {adopting} a Jacobian-free approximation of those derivatives.
In this paper, we analyze the novel method with respect to order of consistency and stability. We show that the linear CFL number varies significantly with the number of derivatives used. Results are verified numerically on several representative testcases.

\end{abstract}


\section{Introduction}\label{sec:intro}
In this work, we present a novel discretization method for the numerical approximation of one-dimensional hyperbolic conservation laws on domain $\Omega \subset \mathbb R$,
\begin{align}\label{eq:1DConsLaw}
w_t + {f}({w})_x &= 0 \, , \qquad \qquad \text{on} \quad (x,t) \in \Omega \times (0,T_{\text{end}}] \, , \\
{w}(x,0) &= {w}_0(x) \, . \notag
\end{align}
Our primary interest is on temporal integration. In recent years, there has been quite some progress on the further development of the \emph{multiderivative} paradigm for temporal integration, see, e.g.,  \cite{TC10,Seal2015b,2021_Gottlieb_EtAl,SSJ2017,Seal13} and the references therein. Assume that one is given a scalar ODE, e.g., 
\begin{align}\label{eq:ODEsystem}
{y}'(t) &= \Phi({y})
\end{align}
for some flux function $\Phi$. Multiderivative schemes make use of not only $\Phi$, but also of the quantities $y''(t) \equiv\Phi'(y)\Phi$, $y'''(t) \equiv \ldots$ and so on. Using this approach, one can derive stable, high-order and storage-efficient schemes very easily \cite{SealSchuetzZeifang21}. This can be extended to partial differential equations (PDEs) with a time-component, such as Eq.~\eqref{eq:1DConsLaw}, depending on the method either directly through the method of lines-discretization \cite{SSJ2017} or through a Lax-Wendroff procedure, see, e.g., \cite{CarrilloPares2019,CarrilloParesZorio2021,LaxWend1960,LiDu2016,Qiu08,ZorioEtAl}. The Lax-Wendroff method expresses temporal derivatives of the unknown function $w$ in terms of the fluxes through the Cauchy-Kowalevskaya procedure. 
As an example, we consider -- for simplicity given that $f$ is scalar -- the second time-derivative of $w$. Due to Eq.~\eqref{eq:1DConsLaw}, there holds 
\begin{align}
 \label{eq:wttoriginal}
 w_{tt} = -(f(w)_x)_t = -(f(w)_t)_x,
\end{align}
and 
\begin{align*}
 f(w)_t = f'(w) w_t = - f'(w) f(w)_x;
\end{align*}
hence 
\begin{align}
  \label{eq:wtt}
  w_{tt} = \left(f'(w) f(w)_x\right)_{x} = 2f'(w)f''(w)w_x^2 + f'(w)^2 w_{xx}. 
\end{align}
Already at this stage, one can see that this approach is very tedious as it necessitates highly complex symbolic calculations.

Still, the potential LW-methods bear is very well recognized among researchers. Over the last two decades, plenty of authors have put effort into developing high-order variants of the LW-method for nonlinear systems. Particularly the ADER (Arbitrary order using DERivatives) methods, see, e.g.,  \cite{PnPm0,2008_Dumbser_Enaux_Toro,axioms7030063,Schwartzkopff2002,Titarev2002,2005_Titarev_Toro} and the references therein, gained a lot of interest. Also, higher-order extensions of the LW-method using WENO and discontinuous Galerkin (DG) reconstructions were investigated \cite{GuoQiuQiu15,MultiDerHDG2015,2011_Lu_Qiu,Qiu08,Qiu2005,QiuShu03}. 

Our essential intent of this paper is to make explicit multistage multiderivative solvers more accessible as a means to solve PDEs. Although such solvers have been theoretically studied since the early 1940's (see \cite{Seal13} for an extensive review), the schemes have not been put much to practice, which is most likely due to the necessary cumbersome calculation of flux derivatives. 
In \cite{CarrilloPares2019}, Carrillo and Par\'es have, based on the earlier work~\cite{ZorioEtAl}, developed the \emph{compact approximate Taylor} (CAT) method to circumvent having to symbolically compute flux derivatives. {By means of an automatic procedure, the higher-order temporal derivatives of $w$, such as in Eq.~\eqref{eq:wtt}, are approximated.} Their work is based on Taylor methods, i.e., time integration is given by 
\begin{equation*}\label{eq:TaylorExpansionTime}
{w}(x,t^{n+1}) = {w}(x,t^{n}) + \sum\limits_{k=1}^{\mathtt{r}} \frac{{\Delta t}^{k}}{k!}\partial^{k}_{t} {w}(x,t^{n}) + \mathcal{O}({\Delta t}^\mathtt{r+1}) \, .
\end{equation*}
In this work, we extend their approach to more general multiderivative integration methods, more precisely, to multiderivative Runge-Kutta (MDRK) methods.

The paper is structured in the following manner: In Sect.~\ref{sec:MDRK} multiderivative Runge-Kutta (MDRK) time integrators for ODEs are introduced, given that they form the central mechanism of this work. Thereafter, in Sect.~\ref{sec:MDRKConsLaw} we shortly revisit the Jacobian-free approach of the CAT method and introduce the explicit Jacobian-free MDRK solver for hyperbolic conservation laws, termed MDRKCAT. After describing the numerical scheme, in Sect.~\ref{sec:consistency} we prove consistency, and in Sect.~\ref{sec:vonNeumann} analyze linear stability. Via several numerical cases we verify and expand on the theoretical results in Sect.~\ref{sec:numres}. At last, we draw our conclusions and discuss future perspectives in Sect.~\ref{sec:conclusion}.

\section{Explicit multiderivative Runge-Kutta solvers}\label{sec:MDRK}
We start by considering the system of ODEs defined by Eq.~\eqref{eq:ODEsystem} in which $\Phi$ is a function of the solution variable ${y} \in \mathbb{R}^{\mathtt{m}}$. In order to apply a time-marching scheme, we discretize the temporal domain with a fixed timestep $\Delta t$ by iterating $N$ amount of steps such that $\Delta t = T_\text{end} / N$. Consequently, we define the time levels by
\begin{equation*}\label{eq:time-discr}
t^n := n\Delta t \, \qquad 0 \leq n \leq N.
\end{equation*}
\begin{remark}\label{Rm:timestep}
 Note that, although the fully space-time-discrete algorithm (Alg.~\ref{alg:MDRKCAT}) \emph{seems} to have a multistep flavour, this is ultimately not the case. It is therefore of no necessity to consider a uniform timestep, which is also demonstrated numerically in Sect.~\ref{sec:numres}. 
\end{remark}
\newcommand{\dotPhi}[1]{\frac{\mathrm d^{#1}}{\mathrm dt^{#1}}{\Phi}}
\newcommand{\dotdPhi}[1]{\frac{\mathrm d^{#1}}{\mathrm dt^{#1}}{\Phi'}}
The central class of time integrators in this work are \emph{explicit} multiderivative Runge-Kutta (MDRK) methods. These form a natural generalization of classical explicit Runge-Kutta methods by adding extra temporal derivatives of $\Phi(w)$. {The additional} time derivatives can be recursively calculated via the chain rule, there holds 
\begin{equation*}\label{eq:MDRK-timeDerTensors}
\dotPhi{k}\left({y}\right) = \frac{\mathrm d^{k-1}}{\mathrm dt^{k-1}}\left(\Phi'(y) \dot y\right) = 
\frac{\mathrm d^{k-1}}{\mathrm dt^{k-1}}\left(\Phi'(y) \Phi(y)\right).
\end{equation*}
For a more detailed description, we refer to \cite{Seal13}.  To present our ideas, let us formally define the MDRK scheme as follows:
\begin{definition}[{\cite[Def. 2]{Seal13}}]\label{def:MDRK-scheme}
An explicit $q$-th order accurate $\mathtt{r}$-derivative Runge-Kutta scheme using $\mathtt{s}$ stages ($\method{\mathtt{r}}{q}{\mathtt{s}}$) is any method which can be formalized as
\begin{subequations}
\begin{equation*}\label{eq:MDRK-ODE-stages}
\hspace{6.4em} {y}^{n,{\color{BlueGreen}l}} := {y}^n + \sum\limits_{k=1}^{\mathtt{r}} {\Delta t}^{k} \sum\limits_{{\color{VioletRed}\nu}=1}^{{\color{BlueGreen}l}-1} a_{{\color{BlueGreen}l}{\color{VioletRed}\nu}}^{(k)}\dotPhi{k-1}\left({y}^{n,{\color{VioletRed}\nu}} \right) \quad {\color{BlueGreen}l} = 1,\dots,\mathtt{s},
\end{equation*}
where ${y}^{n,{\color{BlueGreen}l}}$ is a stage approximation at time $t^{n,{\color{BlueGreen}l}} := t^n + c_{\color{BlueGreen}l} \Delta t$. The update is given by
\begin{equation*}\label{eq:MDRK-ODE-update}
{y}^{n+1} := {y}^n + \sum\limits_{k=1}^{\mathtt{r}} {\Delta t}^{k} \sum\limits_{{\color{BlueGreen}l}=1}^{\mathtt{s}} b^{(k)}_{\color{BlueGreen}l}\dotPhi{k-1}({y}^{n,{\color{BlueGreen}l}}) \, .
\end{equation*}
\end{subequations}
The given coefficients $a_{{\color{BlueGreen}l}{\color{VioletRed}\nu}}^{(k)}$ and $b^{(k)}_{\color{BlueGreen}l}$ determine the scheme; they are typically summarized in an extended Butcher tableau. 
%By gathering the coefficients in matrices $A^{(k)} \in \mathbb{R}^{\mathtt{s}\times\mathtt{s}}$ and vectors ${b}^{(k)}, {c} \in \mathbb{R}^{\mathtt{s}}$, one can also write the defining equations in block-matrix form, see for example \cite{TC10, TurTur2017}.
\end{definition}
\begin{remark}
 Note that standard Taylor methods can be cast in the framework of Def.~\ref{def:MDRK-scheme} through setting $\mathtt{s}=1$, $a_{11}^{(k)}=0$ and $b_1^{(k)} = 1/k!$ $(k=1,\dots,\mathtt{r})$.
 The multiderivative Runge-Kutta schemes used in this work can be found through their extended Butcher tableaux in Appendix~\ref{app:ButherTableaux}. 
\end{remark}

\begin{remark} 
 The stability regions of the used Runge-Kutta methods are visualized in Fig.~\ref{fig:stability-regions-MDRK}, see \cite{TC10,HaiWan,TurTur2017} for more details. Note that except for $\method{3}{5}{2}$ and $\method{4}{6}{2}$, all schemes contain parts of the imaginary axis.
\end{remark}

\begin{center}
\begin{figure}[h!]
\begin{minipage}{0.55\textwidth}
  \centering
	\ifthenelse{\boolean{compilefromscratch}}{
        \tikzsetnextfilename{stability-regions-MDRK}        
		\begin{tikzpicture}
		\begin{axis}[
			width=\linewidth,
			height=6.7cm,
			xlabel={$\Real(z)$},
			xtick distance=1,
			ylabel={$\Imag(z)$},
			ytick distance=2,
			no markers,
			grid=major,
			major grid style = dotted
			]
					
			\addplot[lightViolet, fill=lightViolet!25!white,fill opacity=0.4] table[
					x expr = {\thisrowno{0}},
					y expr = {\thisrowno{3}}
				] {data/stability_domain_FDRK6-2s.dat};		
			\addplot[blue, fill=blue!25!white,fill opacity=0.4] table[
					x expr = {\thisrowno{0}},
					y expr = {\thisrowno{3}}
				] {data/stability_domain_ThDRK7-3s.dat};
			\addplot[mintGreen, fill=mintGreen!25!white,fill opacity=0.4] table[
					x expr = {\thisrowno{0}},
					y expr = {\thisrowno{3}}
				] {data/stability_domain_ThDRK5-2s.dat};
			\addplot[grassGreen, fill=grassGreen!25!white,fill opacity=0.4] table[
					x expr = {\thisrowno{0}},
					y expr = {\thisrowno{3}}
				] {data/stability_domain_TDRK5-3s.dat};
			\addplot[orange, fill=orange!25!white,fill opacity=0.4] table[
					x expr = {\thisrowno{0}},
					y expr = {\thisrowno{1}}
				] {data/stability_domain_TDRK4-2s.dat};
			\addplot[red, fill=red!25!white,fill opacity=0.4] table[
					x expr = {\thisrowno{0}},
					y expr = {\thisrowno{1}}
				] {data/stability_domain_TDRK3-2s.dat};
			
			\addplot[mintGreen, fill=mintGreen!25!white,fill opacity=0.4] table[
					x expr = {\thisrowno{1}},
					y expr = {\thisrowno{4}}
				] {data/stability_domain_ThDRK5-2s.dat};
			\addplot[mintGreen, fill=mintGreen!25!white,fill opacity=0.4] table[
					x expr = {\thisrowno{2}},
					y expr = {\thisrowno{5}}
				] {data/stability_domain_ThDRK5-2s.dat};
				
			\addplot[grassGreen, fill=grassGreen!25!white,fill opacity=0.4]  table[
					x expr = {\thisrowno{1}},
					y expr = {\thisrowno{4}}
				] {data/stability_domain_TDRK5-3s.dat};
			\addplot[grassGreen, fill=grassGreen!25!white,fill opacity=0.4] table[
					x expr = {\thisrowno{2}},
					y expr = {\thisrowno{5}}
				] {data/stability_domain_TDRK5-3s.dat};
			\addplot[blue, fill=blue!25!white,fill opacity=0.4] table[
					x expr = {\thisrowno{1}},
					y expr = {\thisrowno{4}}
				] {data/stability_domain_ThDRK7-3s.dat};
			\addplot[blue, fill=blue!25!white,fill opacity=0.4] table[
					x expr = {\thisrowno{2}},
					y expr = {\thisrowno{5}}
				] {data/stability_domain_ThDRK7-3s.dat};
			\addplot[lightViolet, fill=lightViolet!25!white,fill opacity=0.4] table[
					x expr = {\thisrowno{1}},
					y expr = {\thisrowno{4}}
				] {data/stability_domain_FDRK6-2s.dat};
			\addplot[lightViolet, fill=lightViolet!25!white,fill opacity=0.4] table[
					x expr = {\thisrowno{2}},
					y expr = {\thisrowno{5}}
				] {data/stability_domain_FDRK6-2s.dat};
	
		\end{axis}
		\end{tikzpicture}
	}
	{
	\includegraphics{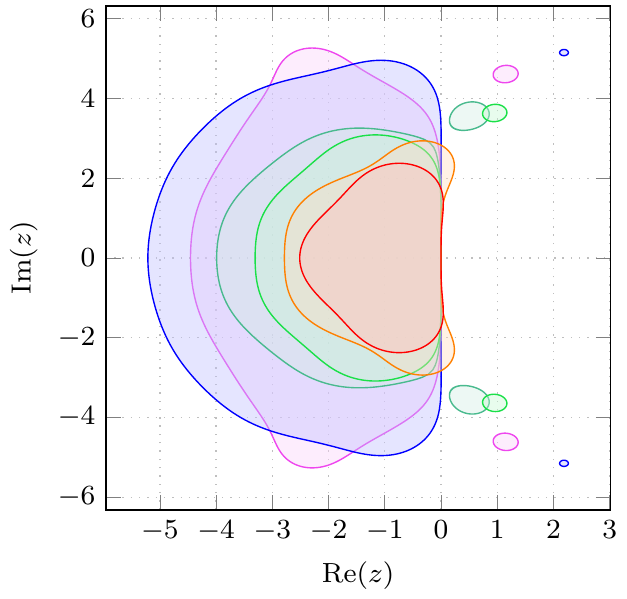}	
	}
\end{minipage}%
\hfill
\begin{minipage}{.45\textwidth}
\centering
\vspace{-1cm}
\resizebox{\textwidth}{!}{%
\begin{tabular}{l|c}
Scheme & $R(z)$ \\ \hline
& \\[\dimexpr-\normalbaselineskip+0.4em]
{\color{red}\rule[2pt]{3mm}{0.8pt}} $\method{2}{3}{2}$ & $\sum_{k=0}^{3} \frac{z^k}{k!}$ \\[1em]
{\color{orange}\rule[2pt]{3mm}{0.8pt}} $\method{2}{4}{2}$ & $\sum_{k=0}^{4} \frac{z^k}{k!}$ \\[1em]
{\color{grassGreen}\rule[2pt]{3mm}{0.8pt}} $\method{2}{5}{3}$ & $\sum_{k=0}^{5} \frac{z^k}{k!} + \frac{z^6}{600}$ \\ [1em]
{\color{mintGreen}\rule[2pt]{3mm}{0.8pt}}  $\method{3}{5}{2}$ & $\sum_{k=0}^{5} \frac{z^k}{k!} + \frac{z^6}{900}$ \\[1em]
{\color{blue}\rule[2pt]{3mm}{0.8pt}} $\method{3}{7}{3}$ & $\sum_{k=0}^{7} \frac{z^k}{k!} + c_8 z^8 + c_9 z^9$ \\[1em]
{\color{lightViolet}\rule[2pt]{3mm}{0.8pt}} $\method{4}{6}{2}$ & $\sum_{k=0}^{6} \frac{z^k}{k!} + \frac{z^7}{6480} + \frac{z^8}{77760}$ 
\end{tabular}
}
\end{minipage}
\caption{Regions of absolute stability $\mathcal{R} := \left\{z \in \mathbb{C} ~|~ |R(z)| \leq 1 \right\}$ and stability functions $R(z)$ for the two-derivative schemes $\method{2}{3}{2}$, $\method{2}{4}{2}$, $\method{2}{5}{3}$~\cite{TC10}, the three-derivative schemes $\method{3}{5}{2}$ and $\method{3}{7}{3}$~\cite{TurTur2017} and the four-derivative scheme $\method{4}{6}{2}$ (see Appendix \ref{app:ButherTableaux}). Except for $\method{3}{5}{2}$ and $\method{4}{6}{2}$, all schemes contain parts of the imaginary axis. Note that in the stability function of $\method{3}{7}{3}$, we have defined $c_8 := \frac{1}{23520} - \frac{\sqrt{2}}{70560}$ and $c_9 := \frac{11}{1481760} - \frac{\sqrt{2}}{246960}$.}
\label{fig:stability-regions-MDRK}
\end{figure}
\end{center}

\section{Multiderivative Runge-Kutta solvers for hyperbolic conservation laws}\label{sec:MDRKConsLaw}
Discretizing the spatial part of the hyperbolic conservation law \eqref{eq:1DConsLaw} necessitates a discretization of the domain $\Omega$. Hence, consider 
\begin{align*} 
 \{x_1, \ldots, x_M\}
\end{align*}
to be a uniform partition of $\Omega$ into $M$ cells of size $\Delta x$. 
A natural extension of Def.~\ref{def:MDRK-scheme} applied to Eq.~\eqref{eq:1DConsLaw} can then be expressed as
\begin{subequations}\label{eq:MDRK-PDE}
\begin{align}
w_i^{n,{\color{BlueGreen}l}} &:= w_i^n - \sum\limits_{k=1}^{\mathtt{r}} {\Delta t}^{k} \sum\limits_{{\color{VioletRed}\nu}=1}^{{\color{BlueGreen}l}-1} a_{{\color{BlueGreen}l}{\color{VioletRed}\nu}}^{(k)}D_x D^{k-1}_t f(w_i^{n,{\color{VioletRed}\nu}})\, , \label{eq:MDRK-PDE-stages} \\
w_i^{n+1} &:= w_i^n - \sum\limits_{k=1}^{\mathtt{r}} {\Delta t}^{k} \sum\limits_{{\color{BlueGreen}l}=1}^{\mathtt{s}} b^{(k)}_{\color{BlueGreen}l}D_x D^{k-1}_t f(w_i^{n,{\color{BlueGreen}l}}),  \label{eq:MDRK-PDE-update}
\end{align}
\end{subequations}
for ${\color{BlueGreen}l} = 1,\dots,\mathtt{s}$; with $D_x$ and $D_t$ being suitable approximations to $\partial_x$ and $\partial_t$ to be explained in the sequel. 
Contrary to the complete Cauchy-Kovalevskaya procedure as outlined for the second derivative in \eqref{eq:wtt}, only \emph{one} time derivative of the solution is transformed into a spatial derivative (see~\eqref{eq:wttoriginal}). 
 
The core focus of this paper is to avoid the explicit use of Jacobians of the flux function $f$. Jacobians of $f$ arise due to the usage of higher temporal derivatives, see, e.g., Eq.~\eqref{eq:wtt}.
In this, we follow the compact approximate Taylor (CAT) approach outlined in \cite{CarrilloPares2019}.
Since the CAT method heavily relies on discrete differentiation, first a small part is devoted to introducing the fundamental notation. Thereafter the method is described and applied to Eqs.~\eqref{eq:MDRK-PDE}.

%%%%%%%%%%%%%%%%%%%%%%%%%%%%%%%%%%%%%%%%%%%%%%%
\subsection{Discrete differentiation}
In this short section, we fix the notation on using finite differencing \cite{1964_Abramowitz_Stegun,QuarteroniNumA,2001_Ralston_Rabinowitz}. 
Considering central differences, the $(2p+1)$-point Lagrangian polynomials are given by 
\begin{equation}\label{eq:LagrangePolyCent}
\hspace{4em} L_{p,j}(\omega) := \prod\limits_{\substack{r=-p \\ r\neq j}}^p \dfrac{\omega - r}{j - r} \qquad j = -p, \dots, p \, .
\end{equation}
It is well-known that these polynomials can be used to interpolate $\varphi: \mathbb{R} \to \mathbb{R}$ in the points $x_{i-p},\dots, x_{i+p}$ through 
\begin{equation}\label{eq:IntPolyCent}
\mathcal{P}_i \varphi(x) := \sum\limits_{j=-p}^p L_{p,j}\left(\frac{x- x_i}{\Delta x}\right)\varphi(x_{i+j}).
\end{equation}
Similarly, from the $2p$-point Lagrangian polynomials
\begin{equation}\label{eq:LagrangePolyNonCent}
\hspace{4em} \ell_{p,j}(\omega) := \prod\limits_{\substack{r=-p+1 \\ r\neq j}}^p \dfrac{\omega - r}{j - r} \qquad j = -p+1, \dots, p \, ,
\end{equation}
(note that the index of the product begins at $r = -p+1$) we obtain the unique polynomial of degree $2p-1$ interpolating $\varphi$ in the points $x_{i-p+1},\dots, x_{i+p}$ through
\begin{equation}\label{eq:IntPolyNonCent}
\mathcal{Q}_i \varphi(x) := \sum\limits_{j=-p+1}^p \ell_{p,j}\left(\frac{x- x_i}{\Delta x}\right)\varphi(x_{i+j}).
\end{equation}
In the sequel, we use a similar notation as in \cite{CarrilloPares2019}: 
\begin{definition}[\cite{CarrilloPares2019}]
For the $k$-th derivative $0 \leq k \leq 2p$, we define the following quantities: 
\begin{alignat*}{2}
\delta^k_{p,j}      &:= L_{p,j}^{(k)}(0), \qquad j & =\ & -p, \ldots, p, \\
\gamma^{k,m}_{p,j}  &:= \ell_{p,j}^{(k)}(m) , \qquad j &\ =\ & -p+1, \ldots, p, \quad m = -p+1, \ldots, p.
\end{alignat*}
Note that the $k$-th derivatives $L_{p,j}^{(k)}(0)$ and $\ell_{p,j}^{(k)}(m)$ are derived analytically from Eq.~\eqref{eq:LagrangePolyCent} and Eq.~\eqref{eq:LagrangePolyNonCent}, respectively.
\end{definition}

Approximate derivatives can thus be derived from
\begin{subequations}\label{eq:IntPolyDer}
\begin{align}
(\mathcal{P}_i \varphi)^{(k)}(x_i) &= \frac{1}{{\Delta x}^k} \sum\limits_{j=-p}^p \delta^k_{p,j}\varphi(x_{i+j}) \, . \label{eq:IntPolyCentDer} \\
(\mathcal{Q}_i \varphi)^{(k)}(x_{i+m}) &= \frac{1}{{\Delta x}^k} \sum\limits_{j=-p+1}^p \gamma^{k,m}_{p,j}\varphi(x_{i+j}) \, , \label{eq:IntPolyNonCentDer}
\end{align}
\end{subequations}
with $m = -p+1, \dots, p$. Since we are working in a discrete context, we define the linear operator counterparts of Eq.~\eqref{eq:IntPolyCentDer} and Eq.~\eqref{eq:IntPolyNonCentDer} as
\begin{alignat*}{2}
P^{(k)} \colon  \mathbb{R}^{2p+1} & \rightarrow \mathbb{R}, &\qquad&  \mathbf{v} \mapsto \frac{1}{{\Delta x}^{k}} \sum\limits_{j=-p}^p \delta^k_{p,j}v_j, \\
Q_m^{(k)} \colon\mathbb{R}^{2p} & \rightarrow \mathbb{R}, &\qquad & \mathbf{w}  \mapsto \frac{1}{{\Delta x}^{k}} \sum\limits_{j=-p+1}^p \gamma^{k,m}_{p,j}w_j \, .
\end{alignat*}
\begin{remark}
The spatial index $i$ is neglected for the linear operators as its direct dependency on the node $x_i$ is lost, cf. Eq.~\eqref{eq:IntPolyCent} and Eq.~\eqref{eq:IntPolyNonCent}. Notice also that $Q_m^{(k)}$ takes vectors in $\mathbb{R}^{2p}$ as input, whereas $P^{(k)}$ takes vectors in $\mathbb{R}^{2p+1}$ as input. 
\end{remark}
A non-centered $2p$-point finite difference method to approximate $\partial^k_t w(x_i,t^{n + m})$ for $m = -p+1,\dots, p$ can therefore be written as
\begin{equation*}\label{eq:exampleFD-approximation}
Q_m^{(k)} \mathbf{w}^{\langle n \rangle}_i = \frac{1}{{\Delta t}^k} \sum\limits_{r=-p+1}^p \gamma^{k,m}_{p,j} w_{i}^{n+r} \, ,
\end{equation*}
with vector notation
\begin{equation}\label{eq:stencilVector}
\mathbf{w}^{\langle n \rangle}_i := \begin{pmatrix}
w_{i}^{n-p+1} \\
\vdots \\
w_{i}^{n+p}
\end{pmatrix} \, .
\end{equation}
The angled brackets represent the local stencil function
\begin{equation}\label{eq:stencilFunction}
\langle\cdot\rangle \colon \mathbb{Z} \to \mathbb{Z}^{2p} \colon n \mapsto \begin{pmatrix} n-p+1, & \dots, & n+p \end{pmatrix} ^T
\end{equation}
throughout this paper, and will be considered for both the spatial index $i$ as the temporal index $n$. Note that the position of the angled bracket (top or bottom) determines whether derivation is w.r.t. time (top) or space (bottom). 

To put the scheme into conservation form, in \cite{ZorioEtAl} auxiliary centered coefficients have been introduced. Here, the operators $P^{(k)}$ for $k \geq 1$ are written as differences of new `half-way point' interpolation operators. 
\begin{definition}[\cite{ZorioEtAl}]
Define $\lambda^{k-1}_{p,j}$ via the relations
\begin{subequations}\label{eq:halfwayCoeff}
\begin{align}
\delta^k_{p,p} &=: \lambda^{k-1}_{p,p} \, , \label{eq:halfwayCoeffa}\\
\delta^k_{p,j} &=: \lambda^{k-1}_{p,j} - \lambda^{k-1}_{p,j+1} \, , \quad j = -p+1, \dots, p-1 \, , \label{eq:halfwayCoeffb} \\
\delta^k_{p,-p} &=: -\lambda^{k-1}_{p,-p+1} \, . \label{eq:halfwayCoeffc} 
\end{align}
\end{subequations} 
\end{definition}
\begin{remark} 
The relations given in Eq.~\eqref{eq:halfwayCoeff} make up an overdetermined system, yet provide a unique solution $\lambda^{k-1}_{p,j}$ obtained from Eq.~\eqref{eq:halfwayCoeffa} and Eq.~\eqref{eq:halfwayCoeffb}, see \cite[Theorem~2]{ZorioEtAl}.
\end{remark} 
Notice the shift between $k$ and $k-1$. This is justified because we enforce a first order derivative relation for the approximation $(\mathcal{P}_i \varphi)^{(k)}(x_i)$ by splitting the operator as 
\begin{equation}\label{eq:SplittingIntPolyCent}
(\mathcal{P}_i \varphi)^{(k)}(x_i) = \frac{\left(\Uplambda^{(k-1)} \varphi\right)(x_{i+1/2})- \left(\Uplambda^{(k-1)} \varphi\right)(x_{i-1/2})}{\Delta x} \, ,
\end{equation}
in which $\Uplambda^{(k-1)}$ is an operator mapping to $\mathbb{P}_{2p-1}$ so that
\begin{equation*}\label{eq:DiffIntPolyCent}
\left(\Uplambda^{(k-1)} \varphi\right)(x_{i+1/2}) := \frac{1}{{\Delta x}^{k-1}} \sum\limits_{j=-p+1}^p \lambda^{k-1}_{p,j} \varphi(x_{i+j})  \, .
\end{equation*}
The linear operator alternative is defined by
\begin{align}
\label{eq:lambdak}
\Lambda^{(k-1)} \colon \mathbb{R}^{2p} &\longrightarrow \mathbb{R}, \qquad \mathbf{v} \mapsto \frac{1}{{\Delta x}^{k-1}} \sum\limits_{j=-p+1}^p \lambda^{k-1}_{p,j}  v_j. 
\end{align}
An overview of all the defined interpolation operators is given in Tbl.~\ref{tab:interpolation-operators}.
\begin{table}[h!]
\centering
\caption{A summary of the defined interpolation operators scaled and shifted to fit the uniform mesh locally at $x_i$. }
\label{tab:interpolation-operators}
\begin{tabular}{c|c||c}
Functional & Linear &  \\ \hline
$(\mathcal{P}_i \varphi)^{(k)}$  & $P^{(k)}\mathbf{v}$ & $k$-th derivative of the Lagrangian interpolation \\
& & polynomial in the nodes $x_{i-p},\dots,x_{i+p}$ \\[0.5em]
$(\mathcal{Q}_i \varphi)^{(k)}$ & $Q_m^{(k)}\mathbf{v}$ & $k$-th derivative of the Lagrangian interpolation  \\
& & polynomial in the nodes $x_{i-p+1},\dots,x_{i+p}$  \\[0.5em]
$\left(\Uplambda^{(k-1)} \varphi\right)(x_{i+1/2})$ & $\Lambda^{(k-1)}\mathbf{v}$ & $(k-1)$-th derivative of a half-way interpolation \\
& & at $x_{i+1/2}$ using the nodes $x_{i-p+1},\dots,x_{i+p}$ \\
& & with the difference coefficients defined by Eq.~\eqref{eq:halfwayCoeff} \\
\end{tabular}
\end{table}

%%%%%%%%%%%%%%%%%%%%%%%%%%%%%%%%%%%%%%%%%%%%%%
\subsection{A Jacobian-free MDRK scheme}
With all the building blocks at our disposal, we can now describe how the final class of methods, that we call MDRKCAT, is assembled. Starting from Eq.~\eqref{eq:MDRK-PDE}, we define the conservative updates of the solution via
\begin{subequations}\label{eq:MDRKCAT-consLaw}
	\begin{align}
	w^{n,{\color{BlueGreen}l}}_i &:= w^n_i - \frac{\Delta t}{\Delta x}\big(F_{i+1/2}^{n,{\color{BlueGreen}l}} - F_{i-1/2}^{n,{\color{BlueGreen}l}} \big) \qquad {\color{BlueGreen}l} = 1,\dots,\mathtt{s} \, , \label{eq:MDRKCAT-consLaw-stage} \\
	w^{n+1}_i &:= w^n_i - \frac{\Delta t}{\Delta x}\big(F_{i+1/2}^{n} - F_{i-1/2}^{n} \big) \, , \label{eq:MDRKCAT-consLaw-update}
	\end{align}
\end{subequations}
in which the numerical fluxes are given by,
\begin{subequations}\label{eq:MDRKCAT-numFlux}
	\begin{align}
	F_{i+1/2}^{n,{\color{BlueGreen}l}} &= \sum\limits_{k=1}^{\mathtt{r}}{\Delta t}^{k-1} \sum\limits_{{\color{VioletRed}\nu}=1}^{{\color{BlueGreen}l}-1} a_{{\color{BlueGreen}l}{\color{VioletRed}\nu}}^{(k)} \Lambda^{(0)}(\widetilde{\mathbf{f}}^{{\color{VioletRed}\nu}})^{(k-1)}_{i,\langle 0 \rangle} \qquad {\color{BlueGreen}l} = 1,\dots,\mathtt{s} \, , \label{eq:MDRKCAT-numFlux-stage} \\
	F_{i+1/2}^{n} &= \sum\limits_{k=1}^{\mathtt{r}}{\Delta t}^{k-1} \sum\limits_{{\color{BlueGreen}l}=1}^{\mathtt{s}} b^{(k)}_{\color{BlueGreen}l} \Lambda^{(0)}(\widetilde{\mathbf{f}}^{\color{BlueGreen}l})^{(k-1)}_{i,\langle 0 \rangle} \, . \label{eq:MDRKCAT-numFlux-update}
	\end{align}
\end{subequations}
For the calculation of $\widetilde{\mathbf{f}}^{(k-1)}_{i,\langle 0 \rangle}$ the compact approximate Taylor (CAT) procedure~\cite{CarrilloPares2019} is used and the flux derivatives can be calculated according to Eq.~\eqref{eq:lambdak} by
%In the notation outlined above, the flux $F^{n}_{i+1/2}$ is given by 
%\begin{align*}
%F^{n}_{i+1/2} &:= \sum\limits_{k=1}^{2p} \dfrac{{\Delta t}^{k-1}}{k!}\Lambda^{(0)}\widetilde{\mathbf{f}}^{(k-1)}_{i,\langle 0 \rangle},
%\end{align*}
\begin{equation*}
\Lambda^{(0)}\widetilde{\mathbf{f}}^{(k-1)}_{i,\langle 0 \rangle} := \sum\limits_{j=-p+1}^p \lambda^0_{p,j}\widetilde{f}^{(k-1)}_{i,j} \, .
\end{equation*}
$\widetilde{f}^{(k-1)}_{i,j} \approx \partial^{k-1}_t f(w)(x_{i+j},t^n)$ indicates the local approximations for the time-derivatives of the flux and are given by
\begin{align*}
\hspace{3cm} \widetilde{f}^{(k-1)}_{i,j} &:= Q_{0}^{(k-1)}\left(\bm{\mathfrak{f}_T}\right)^{k-1,\langle n \rangle}_{i,j} \qquad j = -p+1, \dots, p \, .
\end{align*}
They rely on the approximate flux values $\left(\mathfrak{f}_T\right)^{k-1,n+r}_{i,j} \approx f(w(x_{i+j},t^{n+r}))$. In other words, we take the $(k-1)$-st discrete temporal derivative in $x_{i+j}$ using approximate fluxes 
\begin{align*}
\hspace{0.8cm} \left(\mathfrak{f}_T\right)^{k-1,n+r}_{i,j} &:= f\left(w^n_{i+j} + \sum\limits_{m=1}^{k-1}\frac{(r\Delta t)^m}{m!} \widetilde{w}^{(m)}_{i,j}\right),
\end{align*}
for $j,r = -p+1, \dots, p$. The only thing that is left to define are the quantities $\widetilde{w}^{(m)}_{i,j} \approx \frac{\partial^m}{\partial t^m} w(x_{i+j}, t^n)$.  Their approximation makes heavy use of the Cauchy-Kovalevskaya identity {}$\partial^m_t w = -\partial_x\partial^{m-1}_t f(w)$, they are hence approximated by
\begin{align*}
\hspace{3cm} \widetilde{w}^{(m)}_{i,j} &:= - Q_j^{(1)}\widetilde{\mathbf{f}}^{(m-1)}_{i,\langle 0 \rangle} \qquad j = -p+1, \dots, p.
\end{align*}
{Via the described steps, the vectors $\widetilde{\mathbf{f}}^{(k-1)}_{i,\langle 0 \rangle}$ are recursively obtained, see also~\cite{CarrilloPares2019}.}

\begin{definition}
 For a more precise terminology, we define the specific $\mathtt{r}$-derivative, $q$-th order, $\mathtt s$-stage MDRKCAT method as $\methodcat{\mathtt{r}}{q}{\mathtt{s}}$.
\end{definition}
A summary of the $\methodcat{\mathtt{r}}{q}{\mathtt{s}}$ procedure to obtain the stage values is given in Alg.~\ref{alg:MDRKCAT}. 
Note that the flux at the left half-way point is obtained either from a shift of the index, i.e. $F^n_{i-1/2}=F^n_{i-1+1/2}$ or is given by the boundary condition.
\begin{algorithm}[H]
 	\caption{Stages of $\methodcat{\mathtt{r}}{q}{\mathtt{s}}$, an $\mathtt{r}$-derivative, $q$-th order, $\mathtt s$-stage MDRKCAT method }
  	\label{alg:MDRKCAT}
  	\resizebox{\textwidth}{!}{%
  	\begin{minipage}[t]{0.53\linewidth}
	\begin{algorithm2e}[H]
		\underline{Stage solution (${\color{BlueGreen}l} = 2, \dots, \mathtt{s}$)}:\vspace{1mm}\\
 			\For { $j = -p+1$ \KwTo $p$ }
 			{
 				\vspace{1mm}
 				$(\widetilde{f}^{{\color{BlueGreen}l}-1})^{(0)}_{i,j}= f(w^{n,{\color{BlueGreen}l}-1}_{i+j})$
 				\vspace{1mm} 
 			}
 			$F_{i+1/2}^{n,{\color{BlueGreen}l}} = \sum\limits_{{\color{VioletRed}\nu}=1}^{{\color{BlueGreen}l}-1} a_{{\color{BlueGreen}l}{\color{VioletRed}\nu}}^{(1)} \Lambda^{(0)}(\widetilde{\mathbf{f}}^{{\color{VioletRed}\nu}})^{(0)}_{i,\langle 0 \rangle}$\vspace{2mm}\\
 			\For { $k = 2$ \KwTo $\mathtt{r}$ }
 			{
				\vspace{1mm} 					
				% THE LINE BREAK IS REALLY ARBITRARY!!!!
 				Get $(\widetilde{f}^{{\color{BlueGreen}l}-1})^{(k-1)}_{i,j}$ via CAT procedure.\\
 				%\vspace{2mm}
 				$F_{i+1/2}^{n,{\color{BlueGreen}l}} ~{+}{=}~ {\Delta t}^{k-1}\sum\limits_{{\color{VioletRed}\nu}=1}^{{\color{BlueGreen}l}-1} a_{{\color{BlueGreen}l}{\color{VioletRed}\nu}}^{(k)} \Lambda^{(0)}(\widetilde{\mathbf{f}}^{{\color{VioletRed}\nu}})^{(k-1)}_{i,\langle 0 \rangle}$
 			}
 			\vspace{1mm} $w^{n,{\color{BlueGreen}l}}_i = w^n_i - \frac{\Delta t}{\Delta x}\big(F_{i+1/2}^{n,{\color{BlueGreen}l}} - F_{i-1/2}^{n,{\color{BlueGreen}l}} \big)$
 	\end{algorithm2e}
 	\end{minipage}
 	\begin{minipage}[t]{0.615\linewidth}
 	\begin{algorithm2e}[H]
 		\underline{CAT procedure \cite{CarrilloPares2019} ($k = 2, \dots, \mathtt{r}$)}:\vspace{1mm}\\
 			\For { $j = -p+1$ \KwTo $p$ }
 			{
 				\vspace{3mm}
 				$\widetilde{w}^{(k-1)}_{i,j} = - Q_j^{(1)}\widetilde{\mathbf{f}}^{(k-2)}_{i,\langle 0 \rangle}$  \\
				$\hspace{1cm} = - \frac{1}{{\Delta x}} \sum\limits_{r=-p+1}^p \gamma^{1,j}_{p,r}\widetilde{f}^{(k-2)}_{i,r}$ \vspace{2mm} \\
				
				\For { $r = -p+1$ \KwTo $p$ }
 				{ 	
 								
					\vspace{1mm} $\left(\mathfrak{f}_T\right)^{k-1,n+r}_{i,j} = f\left(w^n_{i+j} + \sum\limits_{m=1}^{k-1}\frac{(r\Delta t)^m}{m!} \widetilde{w}^{(m)}_{i,j}\right) $			
 				}
 				\vspace{3mm}
 				$\widetilde{f}^{(k-1)}_{i,j} = Q_{0}^{(k-1)}\left(\bm{\mathfrak{f}_T}\right)^{k-1,\langle n \rangle}_{i,j}$\\
 				$\hspace{9.5mm} =  \frac{1}{{\Delta t}^{k-1}} \sum\limits_{r=-p+1}^p \gamma^{k-1,0}_{p,r} \left(\mathfrak{f}_T\right)^{k-1,n+r}_{i,j} $ 
			}
 	\end{algorithm2e}
 	\end{minipage}
 	}
\end{algorithm}

\section{Consistency analysis}\label{sec:consistency}
In this section, we show that the $\methodcat{\mathtt{r}}{q}{\mathtt{s}}$ methods are consistent. The order of consistency is, as to be expected, the minimum of the underlying Runge-Kutta order ($q$) and the order of the interpolation ($2p$). Let us make the following two important assumptions: 
\begin{assumption}
 We assume both $f$ and $w$ to be smooth functions in $C^{\infty}$. 
 Furthermore, we assume that $\Delta t$ and $\Delta x$ are asymptotically comparable in size, i.e., 
 \begin{align*}
   \mathcal O(\Delta t) = \mathcal O(\Delta x).  
 \end{align*}
\end{assumption}
Throughout this section, we use the following notation to reduce the number of function arguments: 
\[
{\partial_t^{k} f(w)}_{i+j} \equiv \partial_t^{k} f(w(x_{i+j}, t^n)). 
\]
Whenever possible, similar notation is used for other functions. The time index $n$ is only mentioned when necessary. We immediately state the main result and thereafter, in a successive form, deduce the necessary lemmas upon which its proof relies.
\begin{theorem}\label{thm:mdrkcat-accuracy}
The consistency order of an explicit $\methodcat{\mathtt{r}}{q}{\mathtt{s}}$ method is given by $\min(2p,q)$. Here, $q$ is the consistency order of the underlying MDRK method, while the stencil to update $w(x_i,t^n)$ is given by $\{i-p, i-p+1, \ldots, i+p\}$. 
\end{theorem}
\begin{proof}
The proof relies on Lemmas that will be proven in the sequel. In La.~\ref{lem:mdrkcat-numfluxDiff-accuracy}, it is shown that the numerical flux difference gives the correct flux up to an order of $2p$. We can hence substitute the exact solution $w(x,t)$ into Eq.~\eqref{eq:MDRKCAT-consLaw-update}, which immediately gives the requested result due to the fact that the Runge-Kutta update is an integration scheme of order $q+1$:
\begin{align*}
&\hphantom{=} w(x_{i},t^{n+1}) - w(x_{i},t^n) + \frac{\Delta t}{\Delta x}\big(F_{i+1/2}^{n} - F_{i-1/2}^{n} \big) \\
& = w(x_{i},t^{n+1}) - w(x_{i},t^n) + \sum\limits_{k=1}^{\mathtt{r}} {\Delta t}^{k} \sum\limits_{{\color{BlueGreen}l}=1}^{\mathtt{s}} b^{(k)}_{\color{BlueGreen}l}\partial_x\partial^{k-1}_t f(w)_i^{n,{\color{BlueGreen}l}}  + \mathcal{O}({\Delta x}^{2p+1}) \\
& = \mathcal{O}({\Delta t}^{q+1}) + \mathcal{O}({\Delta x}^{2p+1}) \, .
\end{align*}
$~\hfill \square$
\end{proof}
\begin{remark}
Since the convergence order is $\min(2p,q)$, the optimal choice w.r.t. computational efficiency is to set $p = \lceil q/2 \rceil $. Hence, ``$\methodcat{\mathtt{r}}{q}{\mathtt{s}}$'' does not contain the variable $p$.
\end{remark}

\begin{lemma}\label{lem:mdrkcat-numfluxDiff-accuracy}
The update numerical flux \eqref{eq:MDRKCAT-numFlux-update} satisfies
\begin{equation*}
\frac{F_{i+1/2}^{n} - F_{i-1/2}^{n}}{\Delta x} = \sum\limits_{k=1}^{\mathtt{r}}{\Delta t}^{k-1}\sum\limits_{{\color{BlueGreen}l}=1}^{\mathtt{s}} b^{(k)}_{\color{BlueGreen}l}\partial_x\partial^{k-1}_t f(w)_i^{n,{\color{BlueGreen}l}} + \mathcal{O}({\Delta x}^{2p}) \, .
\end{equation*}
An analogous result holds for the stage flux \eqref{eq:MDRKCAT-numFlux-stage}.
\end{lemma}
\begin{proof}
From La.~\ref{lem:cat-accuracy-k=1} and La.~\ref{lem:cat-accuracy-k>1}, we obtain that for $k > 1$ and ${\color{BlueGreen}l} = 1,\dots,\mathtt{s}$ there holds 
\begin{align*}
(\widetilde{f}^{\color{BlueGreen}l})^{(k-1)}_{i,j} = {\partial_t^{k-1} f(w)}_{i+j}^{n,{\color{BlueGreen}l}} &+ \eta_p^{(k-1)} \big(R^{(k-1)}_{f,t}\big)_{i+j}^{n,{\color{BlueGreen}l}} \cdot {\Delta t}^{2p-k+1} \\
&+ \xi_{p,j}^{(k-1)}\big(R^{(k-1)}_{f,x}\big)_{i}^{n,{\color{BlueGreen}l}} \cdot {\Delta x}^{2p-k+1} + \mathcal{O}({\Delta x}^{2p-k+2}) \, ,
\end{align*} 
with $\eta_p^{(k-1)}$ and $\xi_{p,j}^{(k-1)}$ real-valued coefficients; and $R^{(k-1)}_{f,x}$, $R^{(k-1)}_{f,t}$ smooth functions of space and time.
The above formula is put into use by substituting it into the numerical flux \eqref{eq:MDRKCAT-numFlux-update}:
\begin{align*}
&\frac{\Lambda^{(0)}(\widetilde{\mathbf{f}}^{\color{BlueGreen}l})^{(k-1)}_{i,\langle 0 \rangle} - \Lambda^{(0)}(\widetilde{\mathbf{f}}^{\color{BlueGreen}l})^{(k-1)}_{i-1,\langle 0 \rangle}}{\Delta x} \\
& \quad = \frac{\left(\Uplambda^{(0)} \partial_t^{k-1} f(w)\right)_{i+1/2}^{n,{\color{BlueGreen}l}} - \left(\Uplambda^{(0)} \partial_t^{k-1} f(w)\right)_{i-1/2}^{n,{\color{BlueGreen}l}}}{\Delta x} \\
& \quad\quad + {\Delta t}^{2p-k+1}\cdot \eta_p^{(k-1)} \frac{\big(\Uplambda^{(0)} R^{(k-1)}_{f,t}\big)_{i+1/2}^{n,{\color{BlueGreen}l}} - \big(\Uplambda^{(0)} R^{(k-1)}_{f,t}\big)_{i-1/2}^{n,{\color{BlueGreen}l}}}{\Delta x} \\
& \quad\quad + {\Delta x}^{2p-k+1}\cdot \frac{\big(R^{(k-1)}_{f,x}\big)_{i}^{n,{\color{BlueGreen}l}} - \big(R^{(k-1)}_{f,x}\big)_{i-1}^{n,{\color{BlueGreen}l}}}{\Delta x} \underbrace{\sum\limits_{j=-p+1}^p \lambda^0_{p,j}\xi_{p,j}^{(k-1)}}_{\mathcal{O}(1)} \\[-0.7cm]
& \quad\quad + \frac{1}{\Delta x}\mathcal{O}({\Delta x}^{2p-k+2}) \\[0.5em]
& \quad \refEqual{\eqref{eq:SplittingIntPolyCent}} \left(\mathcal{P}_i\partial_t^{k-1} f(w)\right)^{(1)}(x_{i},t^{n,{\color{BlueGreen}l}}) \\[0.5em]
& \quad\quad + {\Delta t}^{2p-k+1}  \cdot \eta_p^{(k-1)}\big(\mathcal{P}_i  R^{(k-1)}_{f,t}\big)^{(1)}(x_{i},t^{n,{\color{BlueGreen}l}}) + \mathcal{O}({\Delta x}^{2p-k+1}) \\[0.5em]
& \quad = {\partial_x\partial_t^{k-1} f(w)}_{i}^{n,{\color{BlueGreen}l}} + \mathcal{O}({\Delta x}^{2p-k+1}) \, .
\end{align*}
Please note that the term $\left(\big(R^{(k-1)}_{f,x}\big)_{i}^{n,{\color{BlueGreen}l}} - \big(R^{(k-1)}_{f,x}\big)_{i-1}^{n,{\color{BlueGreen}l}}\right)/\Delta x$ can be interpreted as a finite difference approximation of the derivative of the smooth function $R^{k-1}_{f,x}$ and therefore remains bounded, i.e., is $\mathcal{O}(1)$.
\newline$~\hfill \square$

%For $k = 1$ most terms disappear because we set $(\widetilde{f}^{{\color{BlueGreen}l}})^{(0)}_{i,j}= f(w)^{n,{\color{BlueGreen}l}}_{i+j}$. The same result ${\partial_x f(w)}_{i}^{n,{\color{BlueGreen}l}} + \mathcal{O}({\Delta x}^{2p})$ is obtained. 
\end{proof}
Before we regard the consistency analysis of the CAT steps in La.~\ref{lem:cat-accuracy-k=1} and La.~\ref{lem:cat-accuracy-k>1}, the follwing identity on the difference coefficients is described.
\begin{lemma}\label{lem:LagrSumLemma}
Consider a local stencil index $j = -p+1, \dots, p$. Then there holds for $k = 1, \dots, 2p-1$:
\[
\hspace{4em} \sum\limits_{r=-p+1}^p \gamma_{p,r}^{k,j}(r-j)^s = k! \delta_{s,k} \, , \quad s=0,\dots,2p-1 \, .
\]
The symbol $\delta_{s,k}$ here represents the Kronecker-delta function.
\end{lemma}
\begin{proof}
We consider a mesh centered around $x_i = 0$ with spatial size $\Delta x = 1$. The operator $\mathcal{Q}_i$ exactly {interpolates} the polynomial function $\varphi(x) := (x-j)^s$ for $s = 0, \dots, 2p-1$ such that Eq.~\eqref{eq:IntPolyNonCent} becomes
\[
(x-j)^s = \mathcal{Q}_i \varphi(x) = \sum\limits_{r=-p+1}^p \ell_{p,r}(x)\varphi(r) = \sum\limits_{r=-p+1}^p \ell_{p,r}(x)(r-j)^s \, .
\]
Deriving the above relation $k$ times in $x$ and thereafter evaluating in $x = j$ gives the result. $\newline~\hfill \square$
\end{proof}

Now we can provide a consistency proof for the CAT procedure in Alg.~\ref{alg:MDRKCAT}. To this purpose we establish the following notation for the exact Taylor approximation in time of order $k \in \mathbb{N}$,
\begin{equation}\label{eq:exactTaylorTime}
T^k_{i,j}(\tau) := \sum_{m=0}^k \frac{\tau^m}{m!}\partial_t^m w(x_{i+j},t^n) \, .
\end{equation}
It is the $k$-th order approximation of $w(x_{i+j},t^{n}+\tau)$. The proof itself is built in a similar fashion as \cite[Theorem~2]{CarrilloPares2019} and \cite[Proposition~1]{ZorioEtAl}.
\begin{lemma}\label{lem:cat-accuracy-k=1}
For $k=1$ the steps of the CAT algorithm (Alg.~\ref{alg:MDRKCAT}, right side) satisfy
\begin{align*}
\widetilde{w}^{(1)}_{i,j} &= {\partial_t w}_{i+j} + \xi_{p,j}^{(1)}\big(R^{(1)}_{w}\big)_{i}\cdot{\Delta x}^{2p-1}  + \mathcal{O}({\Delta x}^{2p}) \, , \\[2pt]
\widetilde{f}^{(1)}_{i,j} &= {\partial_t f(w)}_{i+j} + \eta_p^{(1)} \big(R^{(1)}_{f,t}\big)_{i+j} \cdot {\Delta t}^{2p-1} 
+ \xi_{p,j}^{(1)}\big(R^{(1)}_{f,x}\big)_{i} \cdot {\Delta x}^{2p-1} + \mathcal{O}({\Delta x}^{2p}) \, ,
\end{align*}
with real-valued coefficients $\xi_{p,j}^{(1)}$, $\eta_p^{(1)}$; and smooth functions $\big(R^{(1)}_{w}\big)_{i}$, $\big(R^{(1)}_{f,t}\big)_{i+j}$, $\big(R^{(1)}_{f,x}\big)_{i}$. (Please note again that $(\cdot)_i$ stands for function-evaluation at $x = x_i$.) 
% \begin{align}
% &\xi_{p,j}^{(1)} := \frac{-1}{(2p)!}\sum\limits_{r=-p+1}^p \gamma_{p,r}^{1,j}(r-j)^{2p} \, , \quad\big(R^{(1)}_{w}\big)_{i} := {\partial_x^{2p} f(w)}_{i} \, , \label{eq:coeff-wTilde1} \\
% &\big(R^{(1)}_{f,t}\big)_{i+j} := \frac{\mathrm{d}^{2p}(f\circ T^1_{i,j})}{{\mathrm{d}\tau}^{2p}}(0) \, , \hspace{1cm} \big(R^{(1)}_{f,x}\big)_{i} := {f'(w)}_i\big(R^{(1)}_{w}\big)_{i} \, , \label{eq:coeff_fTilde1}
% \end{align}
% and
% \begin{equation}\label{eq:eta_p}
% \eta_p^{(1)} := \frac{1}{(2p)!}\sum\limits_{r=-p+1}^p \gamma_{p,r}^{1,0}r^{2p}
% \end{equation}
% 
% a coefficient independent of $j$.
\end{lemma}
\begin{proof}
A straightforward computation on $\widetilde{w}^{(1)}_{i,j}$, using the Taylor expansion of $f(w)$,  reveals that 
\begin{align*}
\widetilde{w}^{(1)}_{i,j} &= - Q_j^{(1)}\widetilde{\mathbf{f}}^{(0)}_{i,\langle 0 \rangle}  
	= - \frac{1}{{\Delta x}} \sum\limits_{r=-p+1}^p \gamma_{p,r}^{1,j}\widetilde{f}^{(0)}_{i,r} \\
	&= - \frac{1}{{\Delta x}} \sum\limits_{r=-p+1}^p \gamma_{p,r}^{1,j}f(w(x_{i+r},t^n)) \\
	&= - \frac{1}{{\Delta x}} \sum\limits_{r=-p+1}^p \gamma_{p,r}^{1,j} \left( \sum\limits_{s=0}^{2p} \frac{\left((r-j)\Delta x\right)^s}{s!} {\partial_x^s f(w)}_{i+j}  + \mathcal{O}({\Delta x}^{2p+1})\right)  \\
	&= - \frac{1}{{\Delta x}}\sum\limits_{s=0}^{2p-1}\frac{{\Delta x}^s}{s!} {\partial_x^s f(w)}_{i+j} \left( \sum\limits_{r=-p+1}^p \gamma_{p,r}^{1,j} (r-j)^s \right)  \\
	& \qquad  - \left[ \frac{1}{(2p)!} \left( \sum\limits_{r=-p+1}^p \gamma_{p,r}^{1,j} (r-j)^{2p} \right) {\partial_x^{2p}f(w)}_i + \mathcal{O}({\Delta x}) \right]{\Delta x}^{2p-1}  \\[0.5em]
	& \qquad + \mathcal{O}({\Delta x}^{2p})  \\[0.5em]
	& \refEqual{La. \ref{lem:LagrSumLemma}} -{\partial_x f(w)}_{i+j} + \xi_{p,j}^{(1)}\big(R^{(1)}_{w}\big)_{i}\cdot{\Delta x}^{2p-1} + \mathcal{O}({\Delta x}^{2p}) \\[0.5em]
	&= {\partial_t w}_{i+j} + \xi_{p,j}^{(1)}\big(R^{(1)}_{w}\big)_{i}\cdot{\Delta x}^{2p-1} + \mathcal{O}({\Delta x}^{2p})  \, ,
\end{align*}
in which we define $\xi_{p,j}^{(1)}$ and $\big(R^{(1)}_{w}\big)_{i}$ by 
\begin{align*}
 &\xi_{p,j}^{(1)} := \frac{-1}{(2p)!}\sum\limits_{r=-p+1}^p \gamma_{p,r}^{1,j}(r-j)^{2p} \, , \quad\big(R^{(1)}_{w}\big)_{i} := {\partial_x^{2p} f(w)}_{i}.
\end{align*}

The succeeding step of the method evaluates the flux $f$ in an approximate Taylor series. We find,
\begin{align*}
\left(\mathfrak{f}_T\right)^{1,n+r}_{i,j} &= f\left(w(x_{i+j},t^n) + (r\Delta t) \widetilde{w}^{(1)}_{i,j}\right) 	\\
&\refEqual{\eqref{eq:exactTaylorTime}} f\left(T^1_{i,j}(r{\Delta t}) + (r\Delta t)\xi_{p,j}^{(1)}\big(R^{(1)}_{w}\big)_{i}\cdot{\Delta x}^{2p-1} + \mathcal{O}({\Delta x}^{2p+1})\right) \\[2pt]
&= (f\circ T^1_{i,j})(r{\Delta t}) \\
& \qquad + (f'\circ T^1_{i,j})(r{\Delta t})\cdot\left[(r\Delta t)\xi_{p,j}^{(1)}\big(R^{(1)}_{w}\big)_{i} \right]\cdot{\Delta x}^{2p-1} + \mathcal{O}({\Delta x}^{2p+1}) \\[2pt]
&= (f\circ T^1_{i,j})(r{\Delta t}) \\
& \qquad + (f'\circ T^1_{i,j})(0)\cdot\left[(r\Delta t)\xi_{p,j}^{(1)}\big(R^{(1)}_{w}\big)_{i} \right]\cdot{\Delta x}^{2p-1} + \mathcal{O}({\Delta x}^{2p+1}) \\[2pt]
&= (f\circ T^1_{i,j})(r{\Delta t}) 
 + {f'(w)}_i \left[(r\Delta t)\xi_{p,j}^{(1)}\big(R^{(1)}_{w}\big)_{i} \right]\cdot{\Delta x}^{2p-1} + \mathcal{O}({\Delta x}^{2p+1}) \, . 
\end{align*}
Consequently, we can find for the temporal interpolation of the fluxes:
\begin{align*}
\widetilde{f}^{(1)}_{i,j} &= Q_{0}^{(1)}\left(\bm{\mathfrak{f}_T}\right)^{1,\langle n \rangle}_{i,j}
= \frac{1}{{\Delta t}} \sum\limits_{r=-p+1}^p \gamma_{p,r}^{1,0}\left(\mathfrak{f}_T\right)^{1,n+r}_{i,j} \\
&= \frac{1}{{\Delta t}} \sum\limits_{r=-p+1}^p \gamma_{p,r}^{1,0}\left[\sum\limits_{s=0}^{2p} \left( \frac{(r{\Delta t})^s}{s!}\frac{\mathrm{d}^{s}(f\circ T^1_{i,j})}{{\mathrm{d}\tau}^{s}}(0) \right) + \mathcal{O}({\Delta x}^{2p+1})\right] \\
& \qquad + {f'(w)}_i \underbrace{\left( \sum\limits_{r=-p+1}^p \gamma_{p,r}^{1,0}r \right)}_{=1, \text{~La.~} \ref{lem:LagrSumLemma}} \xi_{p,j}^{(1)}\big(R^{(1)}_{w}\big)_{i}\cdot{\Delta x}^{2p-1} 
+ \mathcal{O}({\Delta x}^{2p}) \\
&=  \frac{1}{{\Delta t}} \sum\limits_{s=0}^{2p-1} \frac{{\Delta t}^s}{s!}\underbrace{\left( \sum\limits_{r=-p+1}^p \gamma_{p,r}^{1,0}r^s \right)}_{=\delta_{1,s}, \text{~La.~} \ref{lem:LagrSumLemma}} \frac{\mathrm{d}^{s}(f\circ T^1_{i,j})}{{\mathrm{d}\tau}^{s}}(0) \\
& \qquad + \underbrace{\frac{1}{(2p)!}\left( \sum\limits_{r=-p+1}^p \gamma_{p,r}^{1,0}r^{2p} \right)}_{=:\eta_p^{(1)}}\underbrace{\frac{\mathrm{d}^{2p}(f\circ T^1_{i,j})}{{\mathrm{d}\tau}^{2p}}(0)}_{=:\big(R^{(1)}_{f,t}\big)_{i+j}} \cdot{\Delta t}^{2p-1} \\[0.5em]
& \qquad + \xi_{p,j}^{(1)}\underbrace{{f'(w)}_i\big(R^{(1)}_{w}\big)_{i}}_{=:\big(R^{(1)}_{f,x}\big)_{i} }\cdot{\Delta x}^{2p-1} + \mathcal{O}({\Delta x}^{2p}) \\[0.5em]
&= \frac{\mathrm{d}(f\circ T^1_{i,j})}{{\mathrm{d}\tau}}(0) + \eta_p^{(1)} \big(R^{(1)}_{f,t}\big)_{i+j} \cdot {\Delta t}^{2p-1} + \xi_{p,j}^{(1)}\big(R^{(1)}_{f,x}\big)_{i} \cdot {\Delta x}^{2p-1} + \mathcal{O}({\Delta x}^{2p}) \\[2pt]
&= {\partial_t f(w)}_{i+j} + \eta_p^{(1)} \big(R^{(1)}_{f,t}\big)_{i+j} \cdot {\Delta t}^{2p-1} + \xi_{p,j}^{(1)}\big(R^{(1)}_{f,x}\big)_{i} \cdot {\Delta x}^{2p-1} + \mathcal{O}({\Delta x}^{2p}).
\end{align*}
 $~\hfill \square$
\end{proof}

Via induction one can generalize this result.
\begin{lemma}\label{lem:cat-accuracy-k>1}
For $k=2, \dots 2p-1$ the steps of the CAT algorithm (Alg.~\ref{alg:MDRKCAT}, right side) satisfy
\begin{align*}
\widetilde{w}^{(k)}_{i,j} &= {\partial_t^k w}_{i+j} + \xi_{p,j}^{(k)}\big(R^{(k)}_{w}\big)_{i}\cdot{\Delta x}^{2p-k}  + \mathcal{O}({\Delta x}^{2p-k+1}) \, , \\[2pt]
\widetilde{f}^{(k)}_{i,j} &= {\partial_t^k f(w)}_{i+j} + \eta_p^{(k)} \big(R^{(k)}_{f,t}\big)_{i+j} \cdot {\Delta t}^{2p-k} \\
& \hspace{2.04cm} + \xi_{p,j}^{(k)}\big(R^{(k)}_{f,x}\big)_{i} \cdot {\Delta x}^{2p-k} + \mathcal{O}({\Delta x}^{2p-k+1}) \, ,
\end{align*}
with 
\begin{align*}
\eta_p^{(k)} &:= \frac{1}{(2p)!}\left( \sum\limits_{r=-p+1}^p \gamma_{p,r}^{k,0}r^{2p} \right), \qquad
\xi_{p,j}^{(k)} := -\sum\limits_{r=-p+1}^p \gamma_{p,r}^{1,j}\xi_{p,r}^{(k-1)} \\  \big(R^{(k)}_{w}\big)_{i} &:= \big({f'(w)}_i\big)^{k-1}{\partial_x^{2p} f(w)}_{i}, \qquad 
\big(R^{(k)}_{f,t}\big)_{i+j} := \frac{\mathrm{d}^{2p}(f\circ T^k_{i,j})}{{\mathrm{d}\tau}^{2p}}(0), \\\big(R^{(k)}_{f,x}\big)_{i} &:= {f'(w)}_i \big(R^{(k)}_{w}\big)_{i}  \, .
\end{align*}
\end{lemma}

\section{von Neumann stability of $\methodcat{\mathtt{r}}{q}{\mathtt{s}}$ methods}\label{sec:vonNeumann}
The original CAT procedure was developed with the intention to create a scheme which linearly reduces back to high-order Lax-Wendroff methods. As a consequence, the scheme is CFL-1 stable for linear equations \cite[Theorem~1]{CarrilloPares2019}. 
In this section, we discuss the stability properties of the $\methodcat{\mathtt{r}}{q}{\mathtt{s}}$ scheme presented in this work. 
\begin{theorem}[MDRK-LW scheme]\label{thm:MDRK-linearReduction}
Explicit $\methodcat{\mathtt{r}}{q}{\mathtt{s}}$ methods for the linear advection flux $f(w) = \alpha w$ reduce to the numerical scheme 
\begin{subequations}\label{eq:MDRK-HighOrderLW}
\begin{align}
w^{n,{\color{BlueGreen}l}}_j &= w^n_j + \sum\limits_{k=1}^{\mathtt{r}}(-1)^k \alpha^k {\Delta t}^k \left( \sum\limits_{{\color{VioletRed}\nu}=1}^{{\color{BlueGreen}l}-1} a_{{\color{BlueGreen}l}{\color{VioletRed}\nu}}^{(k)} P^{(k)}\mathbf{w}^{n,{\color{VioletRed}\nu}}_{\langle j \rangle} \right) \qquad {\color{BlueGreen}l} = 1,\dots,\mathtt{s} \, , \label{eq:MDRK-HighOrderLW-stages} \\
w^{n+1}_j &= w^n_j + \sum\limits_{k=1}^{\mathtt{r}}(-1)^k \alpha^k {\Delta t}^k \left( \sum\limits_{{\color{BlueGreen}l}=1}^{\mathtt{s}} b^{(k)}_{\color{BlueGreen}l} P^{(k)}\mathbf{w}^{n,{\color{BlueGreen}l}}_{\langle j \rangle} \right) \, , \label{eq:MDRK-HighOrderLW-update}
\end{align}
\end{subequations}
with $P^{(k)}\mathbf{w}^{n,{\color{BlueGreen}l}}_{\langle j \rangle}$ the centered difference approximation of the $k$-th spatial derivative. \end{theorem}
The proof is similar to \cite[Theorem~1]{CarrilloPares2019} and is hence left out. 
In this form it is possible to perform a von Neumann stability analysis (see for example \cite{QuarteroniNumA,2004_Strikwerda}). That is, we fill in the Fourier mode $w^n_j = \mathcal{W}(t^n)e^{\mathrm{i}\kappa x_j}$ with wave number $\kappa \in \mathbb{Z}$ and search for the the amplification factors via the relations
\[ 
\mathcal{W}(t^{n,{\color{BlueGreen}l}}) := g^{n,{\color{BlueGreen}l}}(\kappa) \mathcal{W}(t^n) \quad \text{and} \quad \mathcal{W}(t^{n+1}) := g(\kappa)\mathcal{W}(t^n) \, .
\]
Doing so gives an additional recurrence relation.
\begin{proposition}
The amplification factors obtained from a von Neumann analysis on the MDRK-LW scheme Eq.~\eqref{eq:MDRK-HighOrderLW} is defined by the recurrence relations
\begin{subequations}\label{eq:MDRK-amplificationFactors}
\begin{align}
g^{n,{\color{BlueGreen}l}}(\kappa) &= 1 + \sum\limits_{k=1}^{\mathtt{r}}(-1)^k \sigma^k ~\mathtt{P}^{(k)}(\kappa) \left( \sum\limits_{{\color{VioletRed}\nu}=1}^{{\color{BlueGreen}l}-1} a_{{\color{BlueGreen}l}{\color{VioletRed}\nu}}^{(k)} g^{n,{\color{VioletRed}\nu}}(\kappa) \right) \qquad {\color{BlueGreen}l} = 1,\dots,\mathtt{s} \, ,  \label{eq:MDRK-amplificationFactors-stages} \\
g(\kappa) &= 1 + \sum\limits_{k=1}^{\mathtt{r}}(-1)^k \sigma^k ~\mathtt{P}^{(k)}(\kappa) \left( \sum\limits_{{\color{BlueGreen}l}=1}^{\mathtt{s}} b^{(k)}_{\color{BlueGreen}l} g^{n,{\color{BlueGreen}l}}(\kappa)\right) \, , \label{eq:MDRK-amplificationFactors-update}
\end{align}
\end{subequations}
with wave number $\kappa \in \mathbb{Z}$, $\sigma := \frac{\alpha{\Delta t}}{{\Delta x}}$ the corresponding CFL number and 
\begin{equation}\label{eq:P^k-stability}
\mathtt{P}^{(k)}(\kappa) := \sum_{r=-p}^p \delta^k_{p,r} e^{\mathrm{i}r\kappa{\Delta x}} \, .
\end{equation}
\end{proposition}
The term $\mathtt{P}^{(k)}(\kappa)$ can be interpreted as the $k$-th derivative of the centered $(2p+1)$-point Lagrangian interpolation Eq.~\eqref{eq:IntPolyCentDer} using Fourier basis $\left\{e^{\mathrm{i}l\chi} ~|~ l \in \mathbb{Z} \right\}$ with the grid frequency $\chi = \kappa \Delta x$.

In Tbl.~\ref{tab:cfl-ampFactor-MDRKschemes} we display the amplification factors $g(\kappa)$ for the considered MDRK schemes summarized earlier in Fig.~\ref{fig:stability-regions-MDRK}. Of main interest w.r.t. to these amplification factors is to obtain a critical CFL value $\sigma^* \in \mathbb{R}^+$ such that 
\[
|g(\kappa)| \leq 1 \quad \Leftrightarrow \quad \frac{\alpha{\Delta t}}{{\Delta x}} \leq \sigma^* \, .
\]
We have used the Symbolic Math Toolbox in \texttt{MATLAB} \cite{2020_matlabsymbolic} along with a bisection method on the CFL variable $\sigma$ in Eq.~\eqref{eq:MDRK-amplificationFactors-update} to numerically obtain a $\sigma^*$:
\begin{itemize}
\item In our approach we take $\Delta x = 1$; only the frequency of $g(\kappa)$ is influenced by $\Delta x$, there is no change in absolute value. Assume we would compare mesh sizes $\Delta x$ and $ \Delta \tilde x$, then
\[
\mathtt{P}^{(k)}(\kappa) = \sum\limits_{r=-p}^p \delta^k_{p,r} e^{\mathrm{i}r\kappa{\Delta x}} = \sum\limits_{r=-p}^p \delta^k_{p,r} e^{\mathrm{i}r\left(\kappa \frac{\Delta x}{\Delta \tilde x} \right)\Delta \chi} = \mathtt{P}^{(k)}(\widetilde{\kappa}) \, ,
\]
where $\widetilde{\kappa} := \kappa \frac{\Delta x}{\Delta \tilde x}$.  Thus behavior of $g(\kappa)$ is the same up to a recalibration of the frequency space.

\item Via the toolbox, the value $\max|g(\kappa)|$ is calculated on a uniform 1000-cell mesh of $\kappa \in [-\pi,\pi]$. This domain suffices for this purpose, since for $\Delta x = 1$ all terms $\mathtt{P}^{(k)}(\kappa)$ in Eq.~\eqref{eq:P^k-stability}  are $2\pi$-periodic. 
%In this manner, the maximum is taken over at least three periods of the functions $g(\kappa)$ given in Tbl.~\ref{tab:cfl-ampFactor-MDRKschemes}.
\end{itemize} 

\begin{table}[h!]
\centering
\caption{Amplification factors $g(\kappa)$ of the considered MDRK schemes (see Appendix \ref{app:ButherTableaux}). The CFL criteria (up to four decimals) obtained from $|g(\kappa)| \leq 1$ are also shown.}
\label{tab:cfl-ampFactor-MDRKschemes}
\resizebox{\textwidth}{!}{%
\begin{tabular}{l|c|c}
Scheme & $g(\kappa)$ & CFL \\ \hline
& &  \\[\dimexpr-\normalbaselineskip+0.4em]
$\methodcat{2}{3}{2}$ & $1 - \sigma~\mathtt{P}^{(1)} + \frac{1}{2}\sigma^2 \left( \frac{2}{3}\mathtt{P}^{(1)}\mathtt{P}^{(1)} + \frac{1}{3}\mathtt{P}^{(2)} \right) - \frac{1}{6} \sigma^3 ~\mathtt{P}^{(1)}\mathtt{P}^{(2)}$ & 1.2954 \\[1em]
$\methodcat{2}{4}{2}$ & $1 - \sigma~\mathtt{P}^{(1)} + \frac{1}{2} \sigma^2 ~\mathtt{P}^{(2)} - \frac{1}{6} \sigma^3 ~\mathtt{P}^{(1)}\mathtt{P}^{(2)}+ \frac{1}{24} \sigma^4 ~\mathtt{P}^{(2)}\mathtt{P}^{(2)}$ & 1.4718 \\[1em]
$\methodcat{2}{5}{3}$ & $1 - \sigma~\mathtt{P}^{(1)} + \frac{1}{2} \sigma^2 ~\mathtt{P}^{(2)} - \frac{1}{6} \sigma^3 ~\mathtt{P}^{(1)}\mathtt{P}^{(2)} + \frac{1}{24} \sigma^4 ~\mathtt{P}^{(2)}\mathtt{P}^{(2)}$ & 1.0619 \\[0.2em]
& $- \frac{1}{120} \sigma^5 ~\mathtt{P}^{(1)}\mathtt{P}^{(2)}\mathtt{P}^{(2)} + \frac{1}{600} \sigma^6 ~\mathtt{P}^{(2)}\mathtt{P}^{(2)}\mathtt{P}^{(2)}$ & \\[1em]
$\methodcat{3}{5}{2}$ & $1 - \sigma~\mathtt{P}^{(1)} + \frac{1}{2} \sigma^2 ~\mathtt{P}^{(2)} - \frac{1}{6} \sigma^3 ~\mathtt{P}^{(3)}+ \frac{1}{24} \sigma^4 ~\mathtt{P}^{(1)}\mathtt{P}^{(3)}$ & 0.4275 \\[0.2em]
& $- \frac{1}{120} \sigma^5 ~\mathtt{P}^{(2)}\mathtt{P}^{(3)} + \frac{1}{900} \sigma^6 ~\mathtt{P}^{(3)}\mathtt{P}^{(3)}$ & \\[1em]
$\methodcat{3}{7}{3}$ & $1 - \sigma~\mathtt{P}^{(1)} + \frac{1}{2} \sigma^2 ~\mathtt{P}^{(2)} - \frac{1}{6} \sigma^3 ~\mathtt{P}^{(3)}+ \frac{1}{24} \sigma^4 ~\mathtt{P}^{(1)}\mathtt{P}^{(3)}$ & 0.2300 \\[0.2em]
& $- \frac{1}{120} \sigma^5 ~\mathtt{P}^{(2)}\mathtt{P}^{(3)} + \frac{1}{720} \sigma^6 ~\mathtt{P}^{(3)}\mathtt{P}^{(3)} - \frac{1}{5040} \sigma^7 ~\mathtt{P}^{(1)}\mathtt{P}^{(2)}\mathtt{P}^{(3)} $ &  \\[0.2em]
& $+ \left(\frac{1}{23520} - \frac{\sqrt{2}}{70560} \right) \sigma^8 ~\mathtt{P}^{(2)}\mathtt{P}^{(3)}\mathtt{P}^{(3)} - \left(\frac{11}{1481760} - \frac{\sqrt{2}}{246960} \right) \sigma^9 ~\mathtt{P}^{(3)}\mathtt{P}^{(3)}\mathtt{P}^{(3)}$ & \\[1em]
$\methodcat{4}{6}{2}$ & $1 - \sigma~\mathtt{P}^{(1)} + \frac{1}{2} \sigma^2 ~\mathtt{P}^{(2)} - \frac{1}{6} \sigma^3 ~\mathtt{P}^{(3)}+ \frac{1}{24} \sigma^4 ~\mathtt{P}^{(4)}$ & 0.8563 \\[0.2em]
& $- \frac{1}{120} \sigma^5 ~\mathtt{P}^{(1)}\mathtt{P}^{(4)} + \frac{1}{720} \sigma^6 ~\mathtt{P}^{(2)}\mathtt{P}^{(4)} - \frac{1}{6480} \sigma^7 ~\mathtt{P}^{(3)}\mathtt{P}^{(4)} + \frac{1}{77760} \sigma^8 ~\mathtt{P}^{(4)}\mathtt{P}^{(4)}$ &
\end{tabular}
}
\end{table}

The critical CFL values $\sigma^*$ up to four decimals are shown in Tbl.~\ref{tab:cfl-ampFactor-MDRKschemes}. 
Notice that the two-derivative schemes improve the linear stability, whereas the other schemes reduce the stability compared to the original CAT method. 
To put this observation into perspective, let us point out that the CAT algorithm uses $2p$ derivatives, and thus is based on high-order Lax-Wendroff methods with $\mathtt{r}$ even. If, only for the sake of discussion, we take \emph{uneven} $\mathtt{r}$, we get another picture. For $\mathtt{r} = 1$, we obtain the forward-time central-space scheme, which is infamous for being unconditionally unstable \cite{LEV} and for odd $\mathtt{r} > 1$, we obtain CFL numbers smaller than one.
%\begin{table}[h!]
%\centering
%\caption{The CFL criteria (up to four decimals) obtained from $|g(\kappa)| \leq 1$ for odd-order $\mathtt{r}$ Lax-Wendroff schemes with $p=r+1$.}
%\label{tab:cfl-oddOrderLW}
%\begin{tabular}{c|ccc}
%$\mathtt{r}$ & $3$ & $5$ & $7$ \\ \hline
%CFL & $0.8660$ & $0.3045$ & $0.9966$
%\end{tabular}
%\end{table}
%Clearly, odd order LW-schemes do not possess the favorable CFL-{\color{red}$1$} condition.
Hence, we can make some important observations:
\begin{itemize}
\item Most importantly we can conclude that a method of lines (MOL) viewpoint is inadequate. Solely regarding the stability regions $\mathcal{R}$ in Fig. \ref{fig:stability-regions-MDRK} would give the idea that the $\methodcat{3}{7}{3}$ scheme provides the best stability. This is clearly not the case since the even-derivative schemes are shown to be better in terms of stability. 
%\item Adding extra stages in the form of a MDRKCAT$q$-$\mathtt{sr}$ scheme does not guarantee improved linear stability; the 4DRK6-2s scheme has a CFL-condition lower than 1. \todo{How can you see that from 4DRK6-2s?}
\item Choosing even order derivatives gives better results than choosing an odd number of derivatives, i.e. the two- and four-derivative schemes show better stability than the three-derivative schemes. This is in very good agreement with the observations on the original CAT method.
\item In contrast to Taylor-methods, where only the number of derivatives can be prescribed, MDRK methods have two free parameters, namely the number of derivatives \emph{and} the number of stages. We observe that stages and derivatives highly influence the stability properties of the MDRKCAT method. This allows the identification of well-suited MDRK schemes and gives more flexibility compared to original Taylor-methods. 
\end{itemize}

\section{Numerical results}\label{sec:numres}
In this section, we show numerical results validating our analytical findings. By means of several continuous test cases ranging from scalar PDEs to the system of Euler equations, we show that the expected orders of convergence are obtained. 
For brevity, we do not include flux limiting techniques to this work; hence, we avoid setups where shock formation occurs. Note that flux limiting can be incorporated in a straightforward way as in \cite{CarrilloPares2019}.

The measure for the accuracy in this section is the scaled $l_1$-error at time $t^N \equiv T_{\text{end}}$
\[
\| \mathbf{w}(T_{\text{end}}) - \mathbf{w}^N \|_1 := \Delta x \sum\limits_{i=1}^{M} |w(x_i, T_{\text{end}}) - w_i^{N} | \, ,
\]
$\mathbf{w}: t \mapsto \mathbf{w}(t)$ being a function of time returning a vector of exact solution values (or a reference solution) in the nodes $x_1,\dots,x_M$, and $\mathbf{w}^N$ being the vector of approximations $w^N_i$ at $t^N$. For systems, the sum of the $l_1$-errors corresponding to the separate solution variables is considered. All displayed convergence plots begin at $M = 8$ cells and double the amount of cells with each iteration. In order to enforce the adopted CFL value $\sigma$, the local eigenvalues $\lambda^n_{\text{eig},i}$ of the Jacobian w.r.t. $w_i^n$ are computed. By means of the relation
\[
\Delta t^{n} := \frac{\sigma \Delta x}{\max_{i}|\lambda^n_{\text{eig},i}|} \, 
\]
the timestep is then computed. As the maximum eigenvalue in the computational domain varies over time, a non-constant timestep is prescribed. This highlights the ability of the novel method to use varying timestep sizes, see Rem.~\ref{Rm:timestep}.

\subsection{Burgers equation}
First, we consider Burgers equation 
\begin{equation*}\label{eq:burgers}
\partial_t w + \partial_x \left(\frac{w^2}{2}\right) = 0,\end{equation*}
with the cosine-wave initial condition with periodic boundary conditions
\begin{equation}\label{eq:cosine}
w(x,0) = \frac{1}{4}\cos(\pi x) \qquad \text{on} \quad x \in [0, 2] \,,
\end{equation}
to certify the accuracy $\min(2p,q)$ obtained in Thm.~\ref{thm:mdrkcat-accuracy}. Since the cosine-wave Eq.~\eqref{eq:cosine} has both positive and negative values, the characteristic lines must cross and a shock is formed at some point. The breaking time of the wave \cite{LEV,2011_Whitham} is at $t^* = \frac{4}{\pi}$. Hence we set the final time to $T_{\text{end}} = 0.8$, well before shock formation.

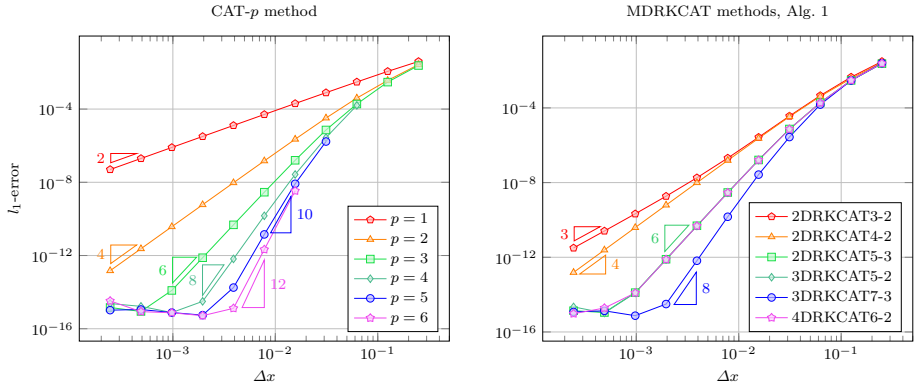
\begin{figure}[h!]
\centering
\begin{subfigure}{.5\textwidth}
  \centering
  \resizebox{\textwidth}{!}{%
  	\ifthenelse{\boolean{compilefromscratch}}{
        \tikzsetnextfilename{lwcat_tfinal0.8_cfl0.5_burgers_cosine}  
		\begin{tikzpicture}
		\begin{loglogaxis}[
			title={CAT-$p$ method},
			xlabel={$\Delta x$},
			ylabel={$l_1$-error},
			grid=major,
			legend entries={$p=1$,$p=2$,$p=3$,$p=4$,$p=5$,$p=6$},
			legend pos=south east,
			cycle list name = rainbow
			]
			\addplot table {data/lwcat_p1_tfinal0.8_cfl0.5_burgers_cosine.dat};
			\addplot table {data/lwcat_p2_tfinal0.8_cfl0.5_burgers_cosine.dat};
			\addplot table {data/lwcat_p3_tfinal0.8_cfl0.5_burgers_cosine.dat};
			\addplot table [comment chars=!] {data/lwcat_p4_tfinal0.8_cfl0.5_burgers_cosine.dat};
			\addplot table {data/lwcat_p5_tfinal0.8_cfl0.5_burgers_cosine.dat};
			\addplot table {data/lwcat_p6_tfinal0.8_cfl0.5_burgers_cosine.dat};
		
			\logLogSlopeTriangle{0.085}{0.07}{0.615}{2}{-1}{red};
			\logLogSlopeTriangle{0.085}{0.07}{0.315}{4}{-1}{orange};
			\logLogSlopeTriangle{0.252}{0.065}{0.275}{6}{-1}{grassGreen};
			\logLogSlopeTriangle{0.333}{0.058}{0.25}{8}{-1}{mintGreen};
			\logLogSlopeTriangle{0.572}{0.055}{0.355}{10}{1}{blue};
			\logLogSlopeTriangle{0.5}{0.06}{0.11}{12}{1}{lightViolet};
		\end{loglogaxis}
		\end{tikzpicture}	
	  }
	  {
	  \includegraphics{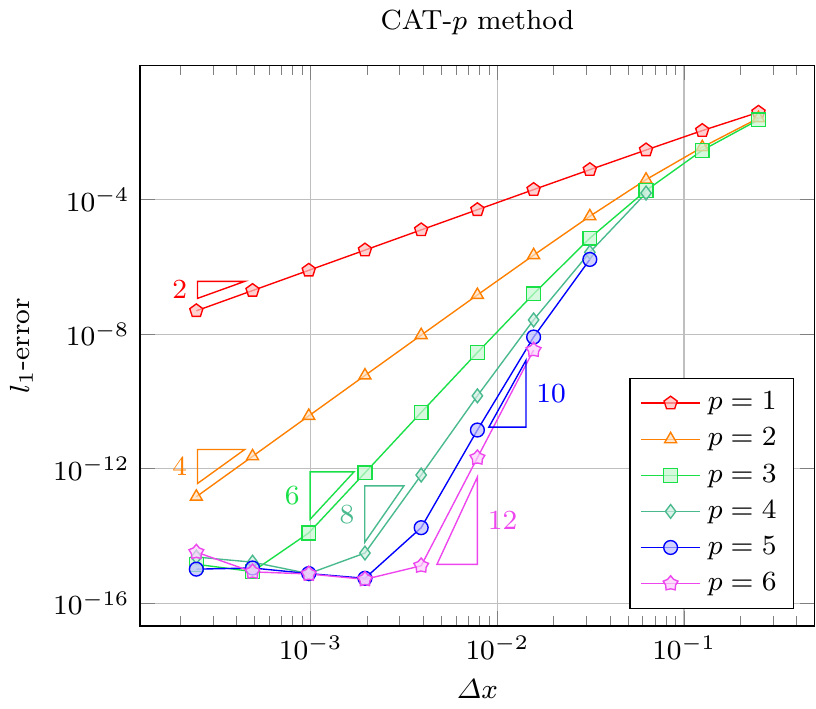}
	  }
	}
\end{subfigure}%
\begin{subfigure}{0.5\textwidth}
  \centering
  \resizebox{\textwidth}{!}{%
  \ifthenelse{\boolean{compilefromscratch}}{
        \tikzsetnextfilename{mderkcat_tfinal0.8_cfl0.5_burgers_cosine}  
		\begin{tikzpicture}
		\begin{loglogaxis}[
			title={MDRKCAT methods, Alg. \ref{alg:MDRKCAT}},
			xlabel={$\Delta x$},
			ylabel={\phantom{$l_1$-error}},
			grid=major,
			legend entries={$\methodcat{2}{3}{2}$,$\methodcat{2}{4}{2}$,$\methodcat{2}{5}{3}$,$\methodcat{3}{5}{2}$,$\methodcat{3}{7}{3}$,$\methodcat{4}{6}{2}$},
			legend pos=south east,
			cycle list name = rainbow
			]
			\addplot table {data/mderkcat_TDRK3-2s_p2_tfinal0.8_cfl0.5_burgers_cosine.dat};
			\addplot table {data/mderkcat_TDRK4-2s_p2_tfinal0.8_cfl0.5_burgers_cosine.dat};
			\addplot table {data/mderkcat_TDRK5-3s_p3_tfinal0.8_cfl0.5_burgers_cosine.dat};
			\addplot table {data/mderkcat_ThDRK5-2s_p3_tfinal0.8_cfl0.5_burgers_cosine.dat};
			\addplot table {data/mderkcat_ThDRK7-3s_p4_tfinal0.8_cfl0.5_burgers_cosine.dat};
			\addplot table {data/mderkcat_FDRK6-2s_p3_tfinal0.8_cfl0.5_burgers_cosine.dat};
	
			\logLogSlopeTriangle{0.085}{0.07}{0.375}{3}{-1}{red};
			\logLogSlopeTriangle{0.17}{0.07}{0.22}{4}{1}{orange};
			\logLogSlopeTriangle{0.33}{0.065}{0.38}{6}{-1}{grassGreen!50!mintGreen};		
			\logLogSlopeTriangle{0.415}{0.06}{0.12}{8}{1}{blue};
		\end{loglogaxis}
		\end{tikzpicture}
	  }
	  {
	  \includegraphics{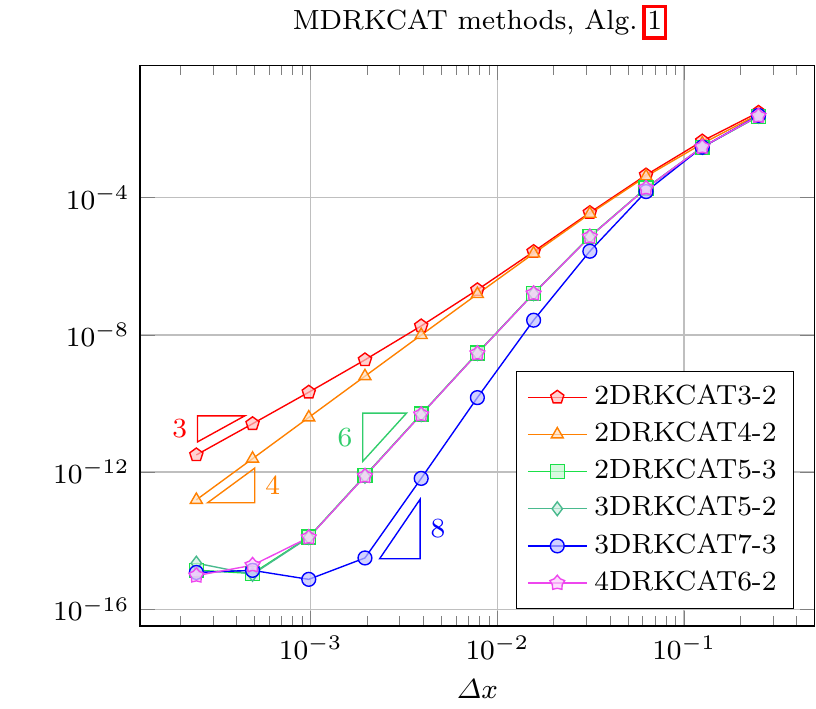}
	  }
	}
\end{subfigure}
\caption{Convergence order of the CAT and explicit MDRKCAT methods applied to Burgers equation on the cosine wave $w_0(x) = \frac{1}{4}\cos(\pi x)$ up to $T_{\text{end}} = 0.8$ with CFL $\sigma= 0.5$. The $\methodcat{2}{5}{3}$, $\methodcat{3}{5}{2}$ and $\methodcat{4}{6}{2}$ schemes behave in a very alike manner. 
For large $\Delta t$ and large $p$, we have observed stability problems for the CAT method. 
Those divergent results of the CAT method have been omitted in the convergence plot. 
%Points not shown on the convergence plot indicate divergence of the method.
}
\label{fig:convPlot_tfinal0.8_cfl0.5_burgers_cosine}
\end{figure}

Using the characteristic lines solution, $l_1$-errors have been calculated for the CAT-$p$ methods ($p = 1, \dots, 6$) and the MDRKCAT methods (Appendix~\ref{app:ButherTableaux}). The results with CFL $\sigma= 0.5$ are visualized in Fig.~\ref{fig:convPlot_tfinal0.8_cfl0.5_burgers_cosine}. All expected convergence orders are obtained; order $2p$ for the CAT-$p$ methods and at least order $q$ of the MDRK schemes. 
The schemes $\methodcat{2}{5}{3}$, $\methodcat{3}{5}{2}$ and $\methodcat{3}{7}{3}$ behave better than expected. This is caused by the fact that the spatial order of accuracy is higher than the temporal one; and spatial errors dominate the overall behavior at least for `large' $\Delta t$. We have observed similar behavior also for other schemes where $2p > q$.
%{\color{magenta} DELETE THE MAGENTA PART. superconvergence \todo{superconvergence is not the right word here!} is observed. 
% This is caused by choosing a higher spatial than temporal order. We have verified this by choosing $p = 6$ using the 3DRKCAT7-3 scheme where we first observe a convergence order of approximately $11$, close to the spatial order $2p$ and for smaller values of $\Delta x$ the overall accuracy is driven back to $q = 7$.}
Given that CAT methods have been designed as a natural generalization of Lax-Wendroff methods with an even-order accuracy, odd-order MDRK schemes take advantage here. We expect this behavior to become more apparent when computing with finer machine precision. Convergence plots such as in Fig.~\ref{fig:convPlot_tfinal0.8_cfl0.5_burgers_cosine} will then manifest as a stretched out S-curve.

Further, we notice that for higher values of $p$ and larger values $\Delta t$, the CAT methods tend to be less stable. Unstable results have been left out in Fig.~\ref{fig:convPlot_tfinal0.8_cfl0.5_burgers_cosine} for $p = 4, 5$ and $6$. Even though CAT methods have been shown to be linearly stable under a CFL-$1$ condition \cite{CarrilloPares2019}, rapid divergence is observed well before shocks are formed for larger values of $p$ at regions where the derivative of the solution $w'(x)$ is large in absolute value. The MDRKCAT methods used in this work do not suffer this fate.

Next, we perform a similar study having the initial solution,
\begin{equation*}\label{eq:expcossin}
w(x,0) = \frac 1 4e^{\cos(\pi x) + \sin(\pi x)} \qquad \text{on} \quad x \in [0, 2] \,,
\end{equation*}
with periodic boundary conditions.
Having a steeper peak than the cosine \eqref{eq:cosine}, the breaking time will be earlier. We find $t^* = \frac{4}{\pi e}$, and choose $T_{\text{end}} = 0.3$ accordingly.

\begin{figure}[h!]
\centering
\begin{subfigure}{.47\textwidth}
  \centering
  \resizebox{\textwidth}{!}{%
  	\ifthenelse{\boolean{compilefromscratch}}{
        \tikzsetnextfilename{sol_ThDRK7-3s_N8192_p4_tfinal0.3_cfl0.5_burgers_expcossin}  
		\begin{tikzpicture}
		\begin{axis}[
			title={$\methodcat{3}{7}{3}$, $\Delta x = 2.44 \cdot 10^{-4}$, $T_{\text{end}} = 0.3$},
			xlabel={$x$},
			ylabel={$w$},
	%		ylabel shift = 0.8em,
	%		ytick distance=0.25,
			grid=major,
			legend entries={$t=0$,$t = 0.3$},
			legend cell align=left,
			legend pos=south west
			]
			
			\addplot [
				black!30!pink,
				domain=0:2,
				samples=100
			] {0.25*exp(cos(deg(pi*x))+sin(deg(pi*x)))};		
			
			\addplot[
				blue,
				mark size=1pt,
				mark options={fill=blue!25!white, fill opacity=0.7},
				mark phase =32,
				mark repeat=96,
				mark=*
			] table[
					x expr = {\thisrowno{0}},
					y expr = {\thisrowno{1}}
				] {data/sol_ThDRK7-3s_N8192_p4_tfinal0.3_cfl0.5_burgers_expcossin.dat};
	
		\end{axis}
		\end{tikzpicture}
	  }
	  {
      \includegraphics{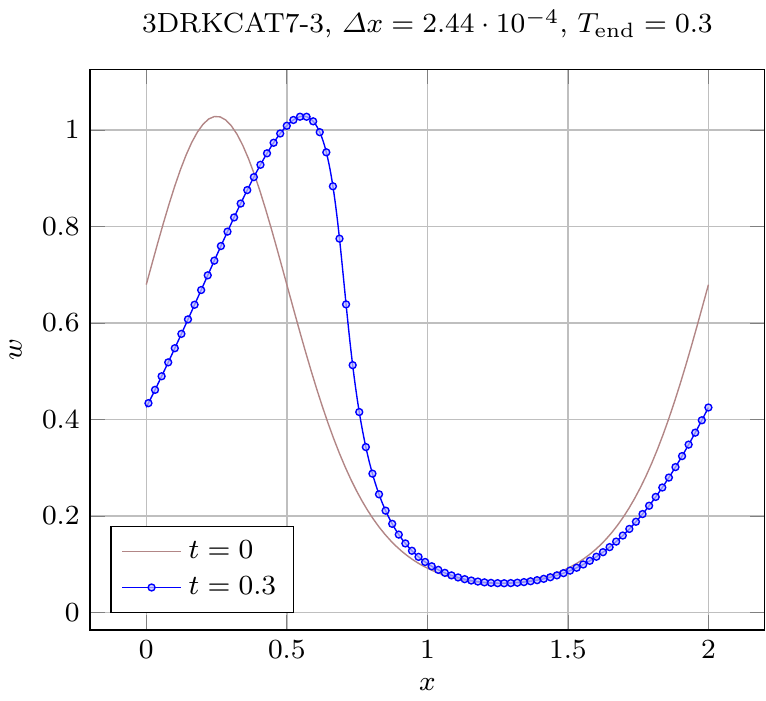}
	  }
	}
\end{subfigure}%
\begin{subfigure}{0.505\textwidth}
  \centering
  \resizebox{\textwidth}{!}{%
 	 \ifthenelse{\boolean{compilefromscratch}}{
        \tikzsetnextfilename{mderkcat_tfinal0.3_cfl0.5_burgers_expcossin}  
		\begin{tikzpicture}
		\begin{loglogaxis}[
			title={MDRKCAT methods, Burgers},
			xlabel={$\Delta x$},
			ylabel={$l_1$-error},
			grid=major,
			legend entries={$\methodcat{2}{3}{2}$,$\methodcat{2}{4}{2}$,$\methodcat{2}{5}{3}$,$\methodcat{3}{5}{2}$,$\methodcat{3}{7}{3}$,$\methodcat{4}{6}{2}$},
			legend pos=south east,
			cycle list name = rainbow
			]
			\addplot table {data/mderkcat_TDRK3-2s_p2_tfinal0.3_cfl0.5_burgers_expcossin.dat};
			\addplot table {data/mderkcat_TDRK4-2s_p2_tfinal0.3_cfl0.5_burgers_expcossin.dat};
			\addplot table {data/mderkcat_TDRK5-3s_p3_tfinal0.3_cfl0.5_burgers_expcossin.dat};
			\addplot table {data/mderkcat_ThDRK5-2s_p3_tfinal0.3_cfl0.5_burgers_expcossin.dat};
			\addplot table {data/mderkcat_ThDRK7-3s_p4_tfinal0.3_cfl0.5_burgers_expcossin.dat};
			\addplot table {data/mderkcat_FDRK6-2s_p3_tfinal0.3_cfl0.5_burgers_expcossin.dat};
	
			\logLogSlopeTriangle{0.085}{0.07}{0.405}{3}{-1}{red};
			\logLogSlopeTriangle{0.17}{0.07}{0.275}{4}{1}{orange};
			\logLogSlopeTriangle{0.33}{0.065}{0.465}{6}{-1}{grassGreen!50!mintGreen};		
			\logLogSlopeTriangle{0.33}{0.06}{0.105}{7}{1}{blue};
		\end{loglogaxis}
		\end{tikzpicture}
	  }
	  {
      \includegraphics{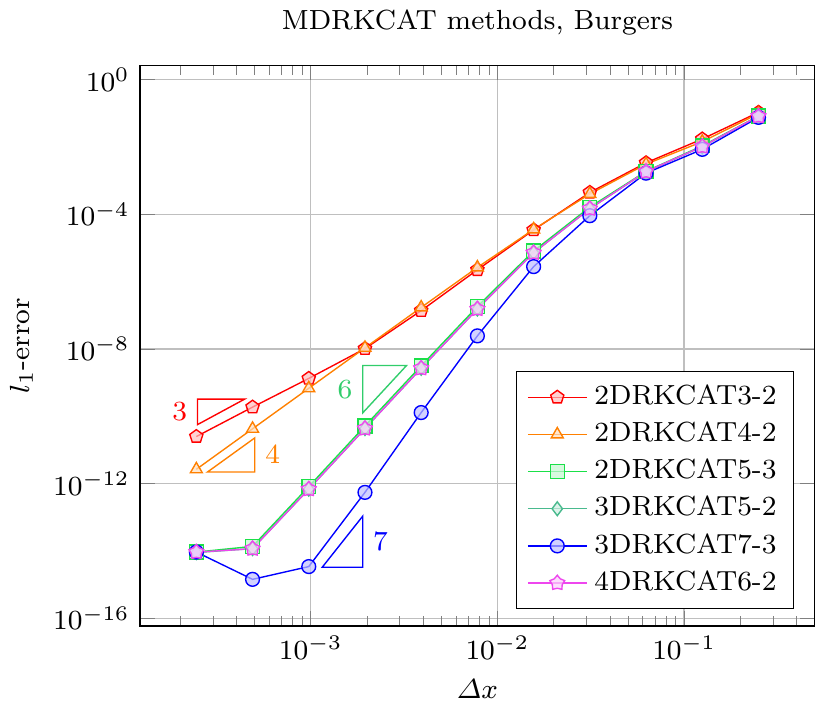}
	  }
	}
\end{subfigure}
\caption{Burgers equation applied to $w_0(x) = \frac 1 4 e^{\cos(\pi x) + \sin(\pi x)}$ up to $T_{\text{end}} = 0.3$ with CFL $\sigma= 0.5$. Left: $\methodcat{3}{7}{3}$ solution using $M = 8192$ cells, less nodal points are shown for better visual distinctness. Right: convergence order of the explicit MDRKCAT methods.}
\label{fig:tfinal0.3_cfl0.5_burgers_expcossin}
\end{figure}
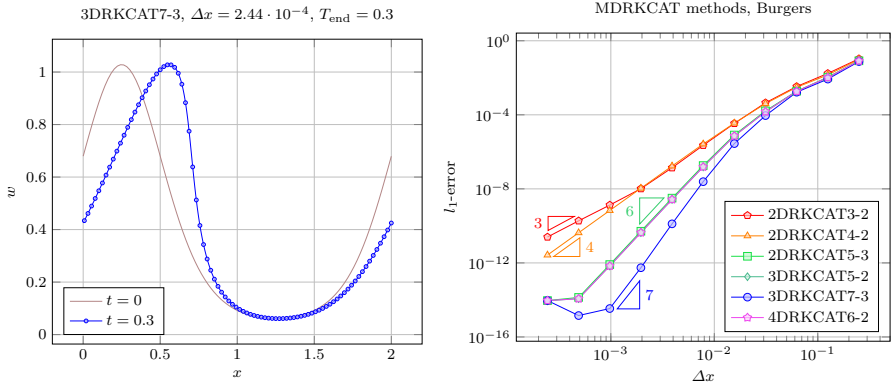

In Fig.~\ref{fig:tfinal0.3_cfl0.5_burgers_expcossin} the final solution at $T_{\text{end}} = 0.3$ is visualized and accuracy is studied in the convergence plots, solely focused on the MDRKCAT method. The behavior is very similar to the previous case, except that the $\methodcat{3}{7}{3}$ scheme is driven back faster to order $7$.

\subsection{Buckley-Leverett equation}
Next, we consider the Buckley-Leverett flux \cite{LEV},
\begin{equation*}\label{eq:BuckleyLeverettFlux}
f(w) = \frac{4w^2}{4w^2 + (1-w)^2} \, .
\end{equation*}
This flux is non-convex and introduces more nonlinearities compared to the Burgers flux. We consider the initial condition
\begin{equation*}\label{eq:downpulse}
w(x,0) = 1 - \frac 3 4\cos^2\left(\frac{\pi}{2} x \right) \qquad \text{on} \quad x \in [-1, 1] \, .
\end{equation*}
The typical Buckley-Leverett profile consist of a shock wave followed directly by a rarefaction wave. We set $T_{\text{end}} = 0.1$ to remain continuous and be able to calculate the exact solution via its characteristics. The solution and convergence plots are visualized in Fig.~\ref{fig:tfinal0.1_cfl0.5_buckley-leverett_cos-downpulse} with CFL $\sigma=0.5$. We notice that the numerical solution tends toward the expected Buckley-Leverett profile. All schemes converge with the expected accuracy in a similar way as for Burgers equation.

\begin{figure}[h!]
\centering
\begin{subfigure}{.47\textwidth}
  \centering
  \resizebox{\textwidth}{!}{%
 	 \ifthenelse{\boolean{compilefromscratch}}{
        \tikzsetnextfilename{sol_ThDRK7-3s_N8192_p4_tfinal0.1_cfl0.5_buckley-leverett_cos-downpulse}  
		\begin{tikzpicture}
		\begin{axis}[
			title={$\methodcat{3}{7}{3}$, $\Delta x =  2.44 \cdot 10^{-4}$, $T_{\text{end}} = 0.1$},
			xlabel={$x$},
			ylabel={$w$},
	%		ylabel shift = 0.8em,
	%		ytick distance=0.25,
			grid=major,
			legend entries={$t=0$,$t = 0.1$},
			legend cell align=left,
			legend pos=south west
			]
			
			\addplot [
				black!30!pink,
				domain=-1:1,
				samples=100
			] {1 - 0.75*cos(deg(0.5*pi*x))^2};		
			
			\addplot[
				blue,
				mark size=1pt,
				mark options={fill=blue!25!white, fill opacity=0.7},
				mark phase= 32,
				mark repeat= 96,
				mark=*
			] table[
					x expr = {\thisrowno{0}},
					y expr = {\thisrowno{1}}
				] {data/sol_ThDRK7-3s_N8192_p4_tfinal0.1_cfl0.5_buckley-leverett_cos-downpulse.dat};
		\end{axis}
		\end{tikzpicture}	
	  }
	  {
      \includegraphics{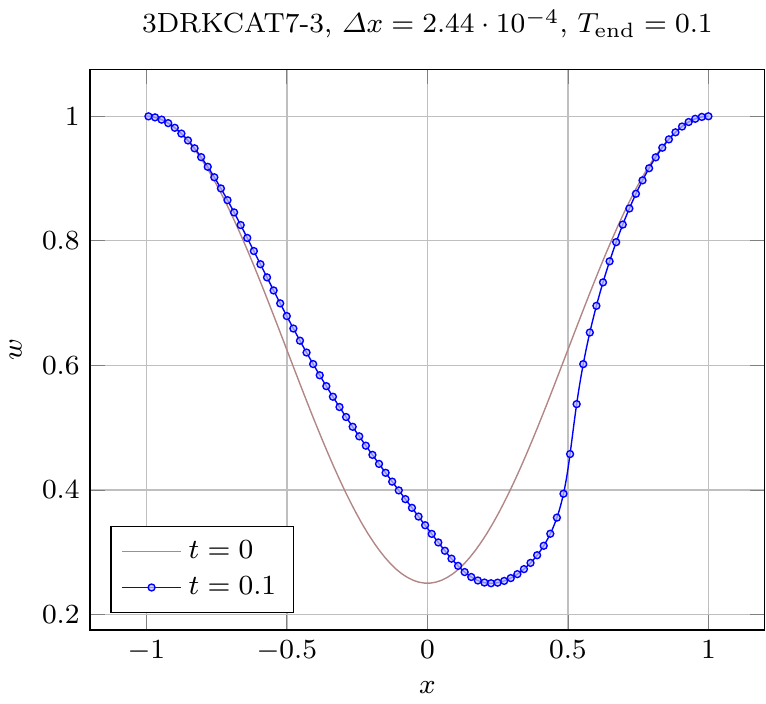}
	  }
	}
\end{subfigure}%
\begin{subfigure}{0.505\textwidth}
  \centering
  \resizebox{\textwidth}{!}{%
 	 \ifthenelse{\boolean{compilefromscratch}}{	
        \tikzsetnextfilename{mderkcat_tfinal0.1_cfl0.5_buckley-leverett_cos-downpulse}
		\begin{tikzpicture}
		\begin{loglogaxis}[
			title={MDRKCAT methods, Buckley-Leverett},
			xlabel={$\Delta x$},
			ylabel={$l_1$-error},
			grid=major,
			legend entries={$\methodcat{2}{3}{2}$,$\methodcat{2}{4}{2}$,$\methodcat{2}{5}{3}$,$\methodcat{3}{5}{2}$,$\methodcat{3}{7}{3}$,$\methodcat{4}{6}{2}$},
			legend pos=south east,
			cycle list name = rainbow
			]
			\addplot table {data/mderkcat_TDRK3-2s_p2_tfinal0.1_cfl0.5_buckley-leverett_cos-downpulse.dat};
			\addplot table {data/mderkcat_TDRK4-2s_p2_tfinal0.1_cfl0.5_buckley-leverett_cos-downpulse.dat};
			\addplot table {data/mderkcat_TDRK5-3s_p3_tfinal0.1_cfl0.5_buckley-leverett_cos-downpulse.dat};
			\addplot table {data/mderkcat_ThDRK5-2s_p3_tfinal0.1_cfl0.5_buckley-leverett_cos-downpulse.dat};
			\addplot table {data/mderkcat_ThDRK7-3s_p4_tfinal0.1_cfl0.5_buckley-leverett_cos-downpulse.dat};
			\addplot table {data/mderkcat_FDRK6-2s_p3_tfinal0.1_cfl0.5_buckley-leverett_cos-downpulse.dat};
	
			\logLogSlopeTriangle{0.085}{0.07}{0.39}{3}{-1}{red};
			\logLogSlopeTriangle{0.17}{0.07}{0.282}{4}{1}{orange};
			\logLogSlopeTriangle{0.247}{0.065}{0.405}{6}{-1}{grassGreen!50!mintGreen};		
			\logLogSlopeTriangle{0.335}{0.06}{0.17}{8}{1}{blue};
		\end{loglogaxis}
		\end{tikzpicture}
	  }
	  {
      \includegraphics{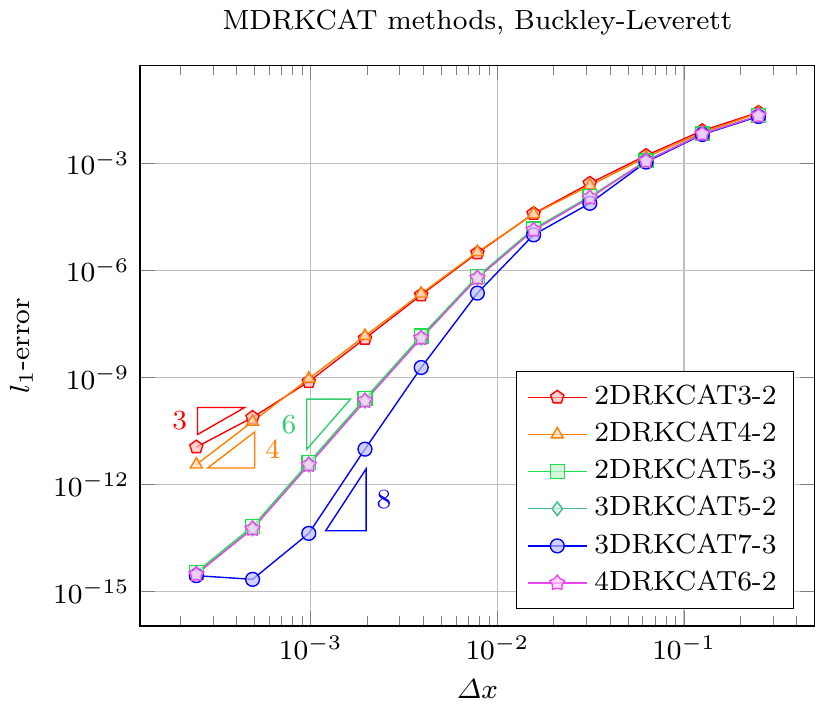}
	  }
	}
\end{subfigure}
\caption{Buckley-Leverett equation applied to $w_0(x) = 1 - \frac 3 4 \cos^2\left(\frac{\pi}{2} x \right)$ up to $T_{\text{end}} = 0.1$ with CFL $\sigma= 0.5$. Left: $\methodcat{3}{7}{3}$ solution using $M = 8192$ cells, less nodal points are shown for better visual distinctness. Right: convergence order of the explicit MDRKCAT methods.}
\label{fig:tfinal0.1_cfl0.5_buckley-leverett_cos-downpulse}
\end{figure}
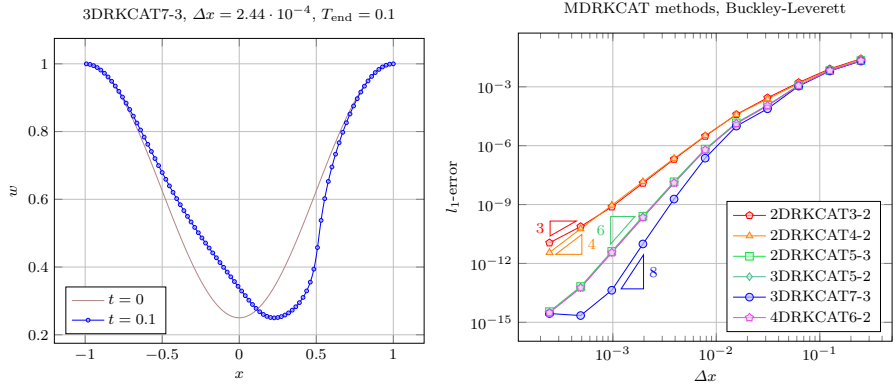

\subsection{One-dimensional Euler equations}
Finally, we consider the Euler equations of gas dynamics
\begin{equation*}\label{eq:systemConsLaws}
\partial_t {w} + \partial_x {f}({w}) = 0 \, ,
\end{equation*} 
in which
\begin{equation*}\label{eq:1DEulerFlux}
{w} = \begin{pmatrix}
\rho \\
\rho u \\
E
\end{pmatrix}\, , \quad
{f}({w}) = \begin{pmatrix}
\rho u \\
\rho u^2 + p \\
u(E + p)
\end{pmatrix} \, ,
\end{equation*}
with $\rho$ the density, $u$ the velocity, $E$ the energy of the system and $p$ the pressure. The system is closed via the equation of state for an ideal gas
\begin{equation*}
p = (\gamma -1)\left(E - \frac{1}{2}\rho u^2 \right) \, ,
\end{equation*}
with $\gamma$ being the ratio of specific heats, assumed to be $1.4$ \cite{LEV,2011_Whitham}.

First we initialize the primitive variables $(\rho, u, p)$ such that the Euler equations describe the linear advection of a density profile. 
To this end, we take
\[
\rho(x,0) = 1 + 0.3\sin(\pi x) \qquad \text{on} \quad x \in [0, 4]\, ,
\]
and set both $u_0(x,0)$ and $p_0(x,0)$ to be one. 
Periodic boundary conditions are used and $T_{\text{end}}$ is set to $0.8$. 
In Fig.~\ref{fig:convPlot_tfinal0.8_cfl0.5_euler_sinewave} the corresponding convergence plots are displayed with CFL $\sigma = 0.5$. Immediately starting from the coarsest meshes the expected convergence orders are obtained.

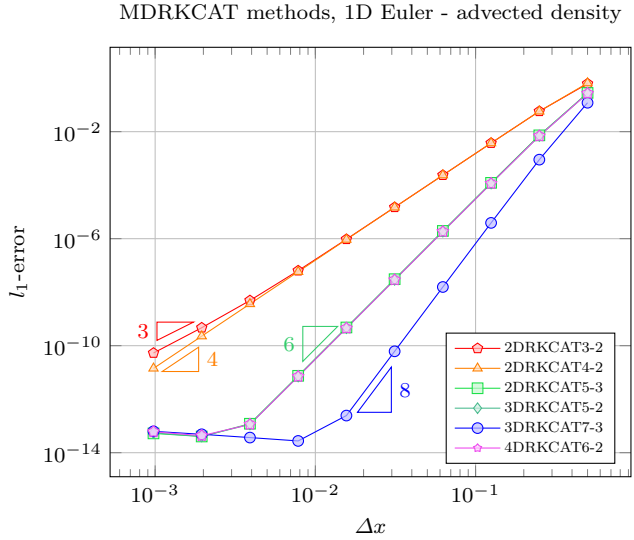
\begin{figure}[h!]
\begin{center}
\ifthenelse{\boolean{compilefromscratch}}{	
        \tikzsetnextfilename{mderkcat_tfinal0.8_cfl0.5_euler-sinewave}
		\begin{tikzpicture}
		\begin{loglogaxis}[
			title={MDRKCAT methods, 1D Euler - advected density},
		    %width=0.72\linewidth,
			%height=6.2cm,
			xlabel={$\Delta x$},
			ylabel={$l_1$-error},
			grid=major,
			legend style={nodes={scale=0.7, transform shape}},
			legend entries={$\methodcat{2}{3}{2}$,$\methodcat{2}{4}{2}$,$\methodcat{2}{5}{3}$,$\methodcat{3}{5}{2}$,$\methodcat{3}{7}{3}$,$\methodcat{4}{6}{2}$, DG-$\mathbb{P}_3$, DG-$\mathbb{P}_4$},
			legend pos=south east,
			cycle list name = rainbow
			]
			\addplot table {data/mderkcat_TDRK3-2s_p2_tfinal0.8_cfl0.5_euler_euler-sinewave.dat};
			\addplot table {data/mderkcat_TDRK4-2s_p2_tfinal0.8_cfl0.5_euler_euler-sinewave.dat};
			\addplot table {data/mderkcat_TDRK5-3s_p3_tfinal0.8_cfl0.5_euler_euler-sinewave.dat};
			\addplot table {data/mderkcat_ThDRK5-2s_p3_tfinal0.8_cfl0.5_euler_euler-sinewave.dat};
			\addplot table {data/mderkcat_ThDRK7-3s_p4_tfinal0.8_cfl0.5_euler_euler-sinewave.dat};
			\addplot[lightViolet, mark=mystar*] table {data/mderkcat_FDRK6-2s_p3_tfinal0.8_cfl0.5_euler_euler-sinewave.dat}; 		
					
		 	\logLogSlopeTriangle{0.09}{0.07}{0.36}{3}{-1}{red};
			\logLogSlopeTriangle{0.17}{0.07}{0.245}{4}{1}{orange};
			\logLogSlopeTriangle{0.37}{0.067}{0.35}{6}{-1}{grassGreen!50!mintGreen};		
			\logLogSlopeTriangle{0.54}{0.065}{0.15}{8}{1}{blue};
		\end{loglogaxis}
		\end{tikzpicture}
	}
	{
    \includegraphics{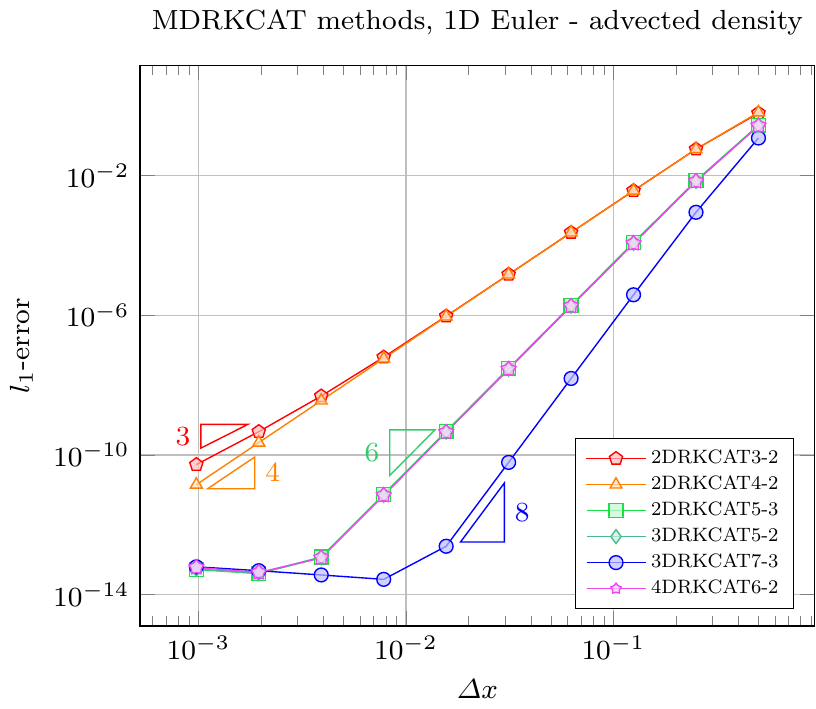}
	}
\caption{Convergence order of the explicit MDRKCAT methods applied to $\rho_0(x) = 1 + 0.3\sin(\pi x)$, $u_0(x) = p_0(x) = 1$ on $x \in [0, 4]$ up to $T_{\text{end}} = 0.8$ with $\sigma= 0.5$. }
\label{fig:convPlot_tfinal0.8_cfl0.5_euler_sinewave}
\end{center}
\end{figure}

Secondly, we consider the initial condition
\begin{equation*}
\label{eq:sine-system}
{w}(x,0) = \frac 1 4 \begin{pmatrix}
3  \\
1  \\
3 
\end{pmatrix}
+ \frac {\sin(\pi x)}{2}
\begin{pmatrix}
1 \\ 1 \\ 1
\end{pmatrix}
\qquad \text{on} \quad x \in [0, 2] \, ,
\end{equation*}
with periodic boundary conditions and $T_{\text{end}} = 0.2$ for ${w}$ to remain continuous. In Fig.~\ref{fig:sol_ThDRK7-3s_N4096_p4_tfinal0.2_cfl0.5_euler_sine-system} the solution is visualized for the $\methodcat{3}{7}{3}$ scheme with $p=4$, CFL $\sigma=0.5$ on $M = 4096$ cells.

In order to inspect accuracy, a reference solution has been computed via a discontinuous Galerkin (DG) method. We have used third-order polynomials in space, and a third-order strong-stability-preserving Runge-Kutta method in time \cite{Gottlieb2001}. The reference computation was executed on $10240$ cells with a CFL number of $\sigma=0.15$. 

Convergence plots have been generated in Fig.~\ref{fig:convPlot_tfinal0.2_cfl0.15_euler_sine-system} using CFL $\sigma=0.15$. All expected orders were obtained. For smaller $\Delta x$, the $l_1$-error converges towards approximately $2\cdot 10^{-11}$, which is the accuracy of the reference DG solution.

Akin to the earlier cases, a CFL value of $\sigma=0.5$ was attempted for the construction of the convergence plots. However not all simulations were stable, more specifically the $\methodcat{3}{7}{3}$ scheme using $p=4$ diverged for $M = 8, 16, 32$ and $64$ cells. % In agreement with Table \ref{tab:cosine-burgers-CAT}, the stability seems to decrease significantly for each incremented value of $p$.

\begin{figure}[h!]
  \centering
  \ifthenelse{\boolean{compilefromscratch}}{	
    \tikzsetnextfilename{sol_ThDRK7-3s_N4096_p4_tfinal0.2_cfl0.5_euler_sine-system_rho}
	\begin{tikzpicture}[baseline]
	\begin{axis}[
		title={$\methodcat{3}{7}{3}$, $\Delta x = 4.88 \cdot 10^{-4}$, $T_{\text{end}} = 0.2$},
		width=\linewidth,
		height=3cm,
		ylabel={$\rho$},
		ylabel shift = 0.8em,
		ytick distance=0.25,
		grid=major,
		legend entries={$t=0$,$t = 0.2$},
		legend cell align=left,
		legend columns = 2,
		legend style={nodes={scale=0.75, transform shape}}
		]
		
		\addplot [
			black!30!pink,
			domain=0:2,
			samples=100
		] {0.75 + 0.5*sin(deg(pi*x))};		
		
		\addplot[
			blue,
			mark size=1pt,
			mark options={fill=blue!25!white, fill opacity=0.7},
			mark phase = 32,
			mark repeat=96,
			mark=*
		] table[
				x expr = {\thisrowno{0}},
				y expr = {\thisrowno{1}}
			] {data/sol_ThDRK7-3s_N4096_p4_tfinal0.2_cfl0.5__euler_sine-system.dat};
%		\addplot[no markers,red] 
%			table[
%				x expr = {\thisrowno{0}},
%				y expr = {\thisrowno{1}}
%			] {data/sol_ThDRK7-3s_N4096_p4_tfinal0.2_cfl0.5__euler_sine-system.dat};

	\end{axis}
	\end{tikzpicture}

    \tikzsetnextfilename{sol_ThDRK7-3s_N4096_p4_tfinal0.2_cfl0.5_euler_sine-system_rhou}
	\begin{tikzpicture}[baseline]
	\begin{axis}[
		width=\linewidth,
		height=3cm,
		ylabel={$\rho u$},
		ytick distance=0.25,
		grid=major
		]
		
		\addplot [
			black!30!pink,
			domain=0:2,
			samples=100
		] {0.25 + 0.5*sin(deg(pi*x))};		
		
		\addplot[
			blue,
			mark size=1pt,
			mark options={fill=blue!25!white, fill opacity=0.7},
			mark phase = 32,
			mark repeat=96,
			mark=*
		] table[
				x expr = {\thisrowno{0}},
				y expr = {\thisrowno{2}}
%			y expr = {\thisrowno{2}/\thisrowno{1}}
			] {data/sol_ThDRK7-3s_N4096_p4_tfinal0.2_cfl0.5__euler_sine-system.dat};

	\end{axis}
	\end{tikzpicture}
    
    \tikzsetnextfilename{sol_ThDRK7-3s_N4096_p4_tfinal0.2_cfl0.5_euler_sine-system_E}
	\begin{tikzpicture}[baseline]
	\begin{axis}[
		width=\linewidth,
		height=3cm,
		xlabel={$x$},
		ylabel={$E$},
		ylabel shift = 0.8em,
		ytick distance=0.25,
		grid=major
		]

		\addplot [
			black!30!pink,
			domain=0:2,
			samples=100
		] {0.75 + 0.5*sin(deg(pi*x))};		
				
		\addplot[
			blue,
			mark size=1pt,
			mark options={fill=blue!25!white, fill opacity=0.7},
			mark phase = 32,
			mark repeat=96,
			mark=*
		] table[
				x expr = {\thisrowno{0}},
				y expr = {\thisrowno{3}}
			] {data/sol_ThDRK7-3s_N4096_p4_tfinal0.2_cfl0.5__euler_sine-system.dat};

	\end{axis}
	\end{tikzpicture}
	}
	{
    \includegraphics{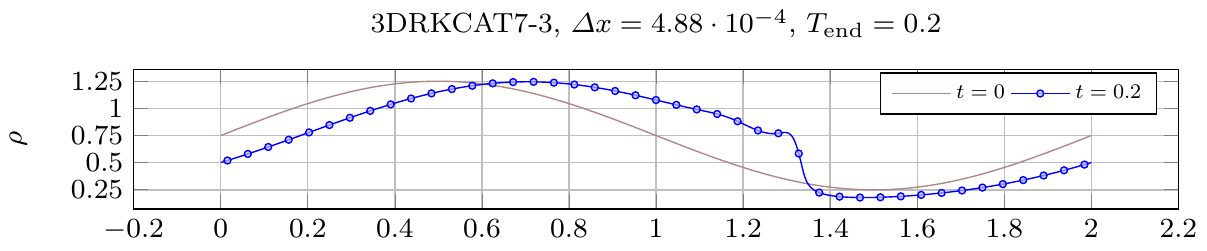}
    \includegraphics{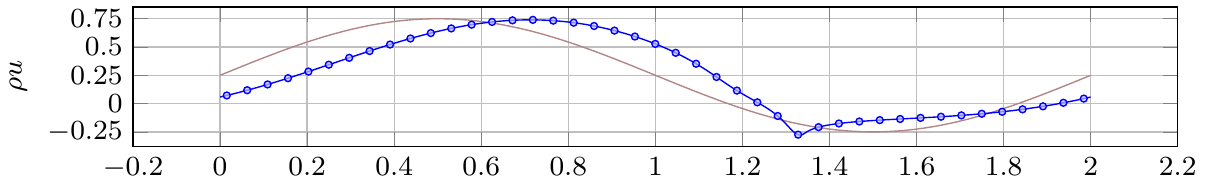}
    \includegraphics{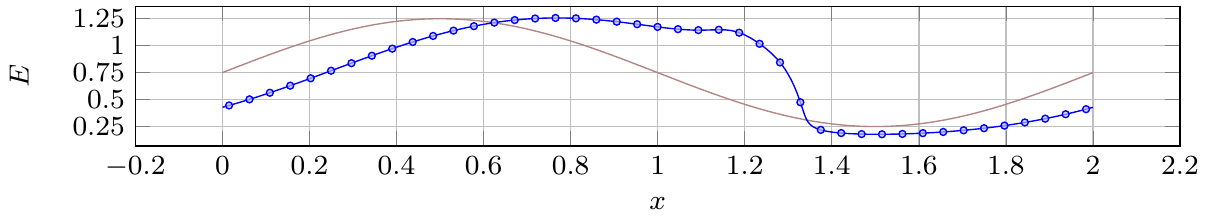}
	}
\caption{MDRKCAT 1D Euler solution of $\rho_0(x) = 0.75 + 0.5\sin(\pi x)$, $(\rho u)_0(x) = 0.25 + 0.5\sin(\pi x)$ and $E_0(x) = 0.75 + 0.5\sin(\pi x)$ up to $T_{\text{end}} = 0.2$ with $\sigma= 0.5$. The $\methodcat{3}{7}{3}$ scheme has been used on $M=4096$ cells; less nodal points are shown for better visual distinctness.}
\label{fig:sol_ThDRK7-3s_N4096_p4_tfinal0.2_cfl0.5_euler_sine-system}
\end{figure}
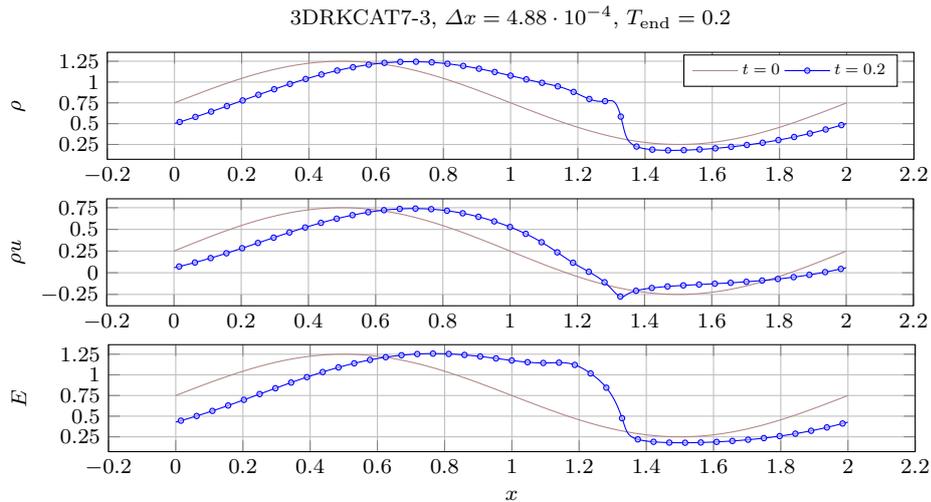

\begin{figure}[h!]
\begin{center}
\ifthenelse{\boolean{compilefromscratch}}{	
        \tikzsetnextfilename{mderkcat_tfinal0.2_cfl0.15_euler_sine-system}
		\begin{tikzpicture}
		\begin{loglogaxis}[
			title={MDRKCAT methods, 1D Euler},
	%	    width=0.5\linewidth,
			xlabel={$\Delta x$},
			ylabel={$l_1$-error},
			grid=major,
			legend style={nodes={scale=0.9, transform shape}},
			legend entries={$\methodcat{2}{3}{2}$,$\methodcat{2}{4}{2}$,$\methodcat{2}{5}{3}$,$\methodcat{3}{5}{2}$,$\methodcat{3}{7}{3}$,$\methodcat{4}{6}{2}$, DG-$\mathbb{P}_3$, DG-$\mathbb{P}_4$},
			legend pos=south east,
			cycle list name = rainbow
			]
			\addplot table {data/mderkcat_TDRK3-2s_p2_tfinal0.2_cfl0.15_euler_sine-system.dat};
	
			\addplot table [comment chars=!] {data/mderkcat_TDRK4-2s_p2_tfinal0.2_cfl0.15_euler_sine-system.dat};
			\addplot table {data/mderkcat_TDRK5-3s_p3_tfinal0.2_cfl0.15_euler_sine-system.dat};
			\addplot table {data/mderkcat_ThDRK5-2s_p3_tfinal0.2_cfl0.15_euler_sine-system.dat};
			\addplot table [comment chars=!] {data/mderkcat_ThDRK7-3s_p4_tfinal0.2_cfl0.15_euler_sine-system.dat};
			\addplot[lightViolet, mark=mystar*] table {data/mderkcat_FDRK6-2s_p3_tfinal0.2_cfl0.15_euler_sine-system.dat}; 		
					
			\addplot[
				mark=halfsquare left*,
				densely dashed
			] table[
					x expr = {(30/2)*\thisrowno{0}},
					y expr = {\thisrowno{1}}
				]{data/dg_convergence.dat};
				
			\addplot[
				magenta,
				mark=halfcircle*,
				mark options={fill=magenta!25!white, fill opacity=0.7,rotate=45},
				densely dashed
			] table[
					x expr = {(30/2)*\thisrowno{0}},
					y expr = {\thisrowno{2}}
				]{data/dg_convergence.dat};

			\logLogSlopeTriangle{0.085}{0.07}{0.36}{4}{-1}{red!50!orange};
			\logLogSlopeTriangle{0.18}{0.05}{0.34}{6}{-1}{grassGreen!50!mintGreen};		
			\logLogSlopeTriangle{0.265}{0.06}{0.165}{7}{1}{blue};
			\logLogSlopeTriangle{0.45}{0.065}{0.295}{4}{1}{magenta};
		\end{loglogaxis}
		\end{tikzpicture}
	}
	{
    \includegraphics{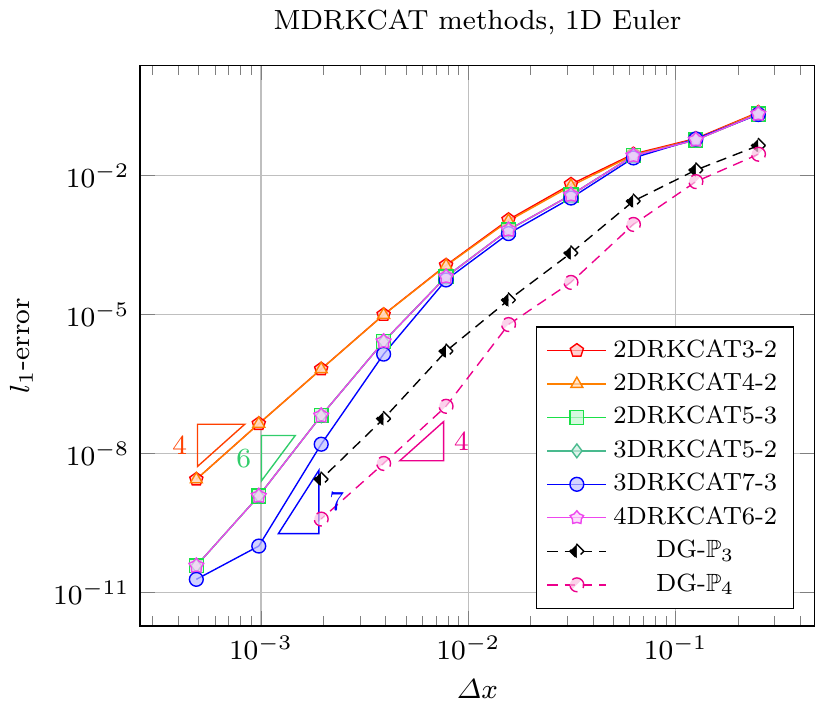}
	}
\caption{Convergence order of the explicit MDRKCAT methods applied to $\rho_0(x) = 0.75 + 0.5\sin(\pi x)$, $(\rho u)_0(x) = 0.25 + 0.5\sin(\pi x)$ and $E_0(x) = 0.75 + 0.5\sin(\pi x)$ up to $T_{\text{end}} = 0.2$ with $\sigma= 0.15$. Very similar behavior can be seen between the $\methodcat{2}{3}{2}$ and the $\methodcat{2}{4}{2}$ schemes. The same can be said for the 5th order MDRKCAT schemes and $\methodcat{4}{6}{2}$. For comparison, $l_1$-errors of a DG code with basis functions in $\mathbb{P}_3$ and $\mathbb{P}_4$ and CFL $\sigma =0.1$ are visualized.}
\label{fig:convPlot_tfinal0.2_cfl0.15_euler_sine-system}
\end{center}
\end{figure}
	
In order to better grasp the efficiency of the MDRKCAT methods, in the same Fig.~\ref{fig:convPlot_tfinal0.2_cfl0.15_euler_sine-system} convergence plots have been generated by means of a DG code that uses polynomial basis functions in $\mathbb{P}_3$ and $\mathbb{P}_4$ for each cell respectively. A fourth order SSP-RK scheme \cite[order 4, p.21]{2002_Spiteri_Ruuth} has been used as time integrator. The same amount of cells $M = 8, \dots, 1024$ has been used as for the MDRKCAT runs, the CFL $\sigma = 0.1$ with a maximum eigenvalue of $1.5$ so that $\Delta t = \Delta x \frac{\sigma}{1.5}$. The $l_1$-errors, computed at cell-midpoints, have been generated relative to the earlier mentioned order three SSP-DG reference solution.

Overall, the MDRKCAT methods compare well with the DG solutions. A large discrepancy can be noticed in the manner at which the expected convergence order is achieved; the MDRKCAT methods gradually head toward order $\min(2p,q)$ with each refinement, whereas the DG solvers achieve convergence already going from 32 to 64 cells. This is to be expected: By definition the MDRKCAT methods only approximate the time derivatives of the flux $\partial^{k-1}_t {f}({w})$. Hence the achieved accuracy is intertwined with the mesh resolution of the problem at hand. For a lower amount of cells $M$ the numerical flux $F^n_{i+1/2}$ at the faces can thus not be an accurate representation, whereas the DG solvers calculate the fluxes at the half-way points $i+1/2$ on the basis of the exact flux ${f}({w})$. As soon as enough cells $M$ are used to finely represent the initial data, full advantage can be taken of the CAT method.
 
Moreover, the difference between the methods should be brought into perspective by studying the amount of effective \textit{spatial} degrees of freedom (DOF) and the effective spatial size that influences the order of accuracy. DG methods make use of numerical integration points on each cell for the integration of the solution variable multiplied with the chosen basis functions \cite{Cockburn2000}. This illustrates why the DG solvers more quickly capture the expected convergence order and why a direct comparison of the DG schemes and the MDRKCAT is difficult in Fig.~\ref{fig:convPlot_tfinal0.2_cfl0.15_euler_sine-system}.
%A comparison is given in Table \ref{tab:effectiveDOF-comparison}.
%\begin{table}[h!]
%\centering
%\caption{A comparison of the effectively used spatial degrees of freedom and the effective spatial size ${\Delta x}_{\text{eff}}$ that play a role on accuracy for DG methods and MDRKCAT schemes.}
%\label{tab:effectiveDOF-comparison}
%\begin{tabular}{c|ccc}
%Method & MDRKCAT$q$-$\mathtt{sr}$ & DG-$\mathbb{P}_3$ & DG-$\mathbb{P}_4$ \\ \hline
%$\text{DOF}_{\text{eff}}$ & $M$ & $4M$ & $5M$ \\
%${\Delta x}_{\text{eff}}$ & $\Delta x$ & $\Delta x / 4$ & $\Delta x / 5$
%\end{tabular}
%\end{table}
The actual amount of spatial DOF used by $\methodcat{\mathtt{r}}{q}{\mathtt{s}}$ schemes is $\lceil q/2 \rceil M$; each node uses it's own local stencil in the calculations. However, as explained in \cite{CarrilloPares2019}, the local stencils are merely a manner to assure that the CAT methods linearly reduce back to Lax-Wendroff schemes. The same accuracy is achieved by the approximate Taylor methods in \cite{ZorioEtAl} of which the CAT procedure is established. Summing up, we can conclude that the novel $\methodcat{\mathtt{r}}{q}{\mathtt{s}}$ schemes compare well with a state-of-the-art DG solver in terms of accuracy.

%Hence the reason that $\text{DOF}_{\text{eff}}$ is set equal to $M$ in Table \ref{tab:effectiveDOF-comparison}.\\
%With this in mind, the DG-$\mathbb{P}_3$ (resp. DG-$\mathbb{P}_4$) convergence lines in Figure \ref{fig:convPlot_tfinal0.2_cfl0.15_euler_sine-system} effectively should be scaled to smaller values.% ${\Delta t}/4$ (resp. ${\Delta t}/5$).
%Moreover, this provides additional clarification as to why the DG solvers more quickly capture the expected convergence order.

\section{Conclusion and outlook}\label{sec:conclusion}
In this paper we have formulated a family of Jacobian-free multistage multiderivative solvers for hyperbolic conservation laws, so-called MDRKCAT methods. Following the Compact Approximate Taylor (CAT) method in \cite{CarrilloPares2019}, instead of computing the exact flux derivative expressions, local approximations for the time derivatives of the fluxes are  obtained recursively.  There are many advantages by virtue of this procedure: no costly symbolic computations are needed; and we hope that many multiderivative Runge-Kutta (MDRK) schemes can now actually be of practical use.

Both theoretically and numerically it is proven that the desired convergence order $\min(2p,q)$ is achieved, $2p$ being the spatial order and $q$ the temporal order. Universally among the different test cases the spatial accuracy is seen to be dominant. A comparison with SSP-DG methods for the Euler equations shows that MDRKCAT methods compare well with state-of-the-art schemes {in terms of accuracy}. 

%A von Neumann analysis interestingly demonstrated that addition of extra stages can improve linear stability, however this is not self-evident. 
A von Neumann analysis revealed that the stability of the MDRKCAT methods depends heavily on the number of stages and the underlying high-order Lax-Wendroff method. The latter one solely utilizes centered differences for the spatial discretization. Consequently, odd-derivative Runge-Kutta schemes seem less adequate in conjuction with the CAT algorithm. 

{In the future, there are two main routes to follow: extend and apply the scheme to more challenging settings and to further examine the stability properties of the novel scheme. Concerning more challenging settings the investigation of}
multidimensional hyperbolic conservation laws with (possibly) unstructured meshes and parabolic PDEs with viscous effects are attractive. In order to accomplish such extensions it might be interesting to combine MDRKCAT methods with DG techniques \cite{SSJ2017}. Presumably, also implicit MDRK schemes need to be considered to take care of the diffusive effects. {A possible starting point could be the implicit variant of the approximate Taylor methods, which have been recently developed for ODEs in~\cite{2020_Baeza_EtAl}.}
Concerning the stability properties of the scheme one could think of exploring more types of MDRK schemes, possibly with SSP properties \cite{Seal2015b,2021_Gottlieb_EtAl}.
Moreover, at the same time, it will be possible to identify more efficient schemes. 

\section*{Declarations}
\textit{Funding} J. Zeifang was funded by the Deutsche Forschungsgemeinschaft (DFG, German Research Foundation) through
the project no. 457811052. The HPC-resources and services used in this work were provided by the VSC (Flemish Supercomputer Center), funded by the Research Foundation - Flanders (FWO) and the Flemish Government.\\
\\
\textit{Conflicts of interest} The authors declare that they have no known competing financial interests or personal relationships that could have appeared to influence the work reported in this paper. \\
\\
\textit{Availability of data and material} The datasets generated and/or analyzed during the current study are available from the corresponding author on reasonable request \url{jeremy.chouchoulis@uhasselt.be}.\\
\\
\textit{Code availability} The code used to generate the results in this work is available upon reasonable request from the corresponding author \url{jeremy.chouchoulis@uhasselt.be}.

\appendix

\section{Butcher tableaux}\label{app:ButherTableaux}
In this section, we show the multiderivative Runge-Kutta methods used in this work through their Butcher tableaux. We use three two-derivative methods taken from \cite{TC10}, see Tbl.~\ref{tab:TDRK3-2s}--\ref{tab:TDRK5-3s}; two three-derivative methods taken from \cite{TurTur2017}, see Tbl.~\ref{tab:ThDRK5-2s}--\ref{tab:ThDRK7-3s}; and one four-derivative method, constructed for this paper, see Tbl.~\ref{tab:FDRK6-2s}. This last scheme has been derived from the idea that it should be of form 
\begin{subequations}\label{eq:specificMDRK}
\begin{equation*}\label{eq:specificMDRK-stages}
y^{n,{\color{BlueGreen}l}} = y^n + \sum\limits_{k=1}^{\mathtt{r}-1} \frac{\left(c_{\color{BlueGreen}l} {\Delta t}\right)^{k}}{k!} \Phi^{(k-1)}\left(y^{n}\right) + {\Delta t}^{\mathtt{r}}\sum\limits_{{\color{VioletRed}\nu}=1}^{{\color{BlueGreen}l}-1} a_{{\color{BlueGreen}l}{\color{VioletRed}\nu}}^{(\mathtt{r})} \Phi^{(\mathtt{r}-1)}\left(y^{n,{\color{VioletRed}\nu}} \right) \, ,
\end{equation*}
for ${\color{BlueGreen}l} = 1,\dots,\mathtt{s}$, with update
\begin{equation*}\label{eq:specificMDRK-update}
y^{n+1} = y^n + \sum\limits_{k=1}^{\mathtt{r}-1} \frac{{\Delta t}^{k}}{k!} \Phi^{(k-1)}\left(y^{n}\right) + {\Delta t}^{\mathtt{r}}\sum\limits_{{\color{BlueGreen}l}=1}^{\mathtt{s}} b^{(\mathtt{r})}_{\color{BlueGreen}l} \Phi^{(\mathtt{r}-1)}(y^{n,{\color{BlueGreen}l}}) \, . \hspace{2em}
\end{equation*}
\end{subequations}
These forms have also been used in \cite{TC10} and \cite{TurTur2017}.

\begin{table}[h!]
\centering
\caption{$\method{2}{3}{2}$ - Third order two-derivative Runge-Kutta scheme using two stages \cite{TC10}.}
\label{tab:TDRK3-2s}
\begin{tabular}{c|cc|cc}
$0$ & $0$   & $0$   & $0$   & $0$ \\
$1$ & $1$   & $0$   & $1/2$ & $0$ \\ \hline
  & $2/3$ & $1/3$ & $1/6$ & $0$
\end{tabular}
\end{table}

\begin{table}[h!]
\centering
\caption{$\method{2}{4}{2}$ - Fourth order two-derivative Runge-Kutta scheme using two stages \cite{TC10}.}
\label{tab:TDRK4-2s}
\begin{tabular}{c|cc|cc}
$0$ & $0$   & $0$   & $0$   & $0$ \\
$1/2$ & $1/2$   & $0$   & $1/8$ & $0$ \\ \hline
  & $1$ & $0$ & $1/6$ & $1/3$
\end{tabular}
\end{table}

\begin{table}[h!]
\centering
\caption{$\method{2}{5}{3}$ - Fifth order two-derivative Runge-Kutta scheme using three stages \cite{TC10}.}
\label{tab:TDRK5-3s}
\begin{tabular}{c|ccc|ccc}
$0$ & $0$   & $0$  & $0$   & $0$ & $0$   & $0$ \\
$2/5$ & $2/5$ & $0$  & $0$ & $2/25$ & $0$ & $0$ \\
$1$ & $1$ & $0$  & $0$ & $-1/4$ & $3/4$ & $0$ \\  \hline
  & $1$ & $0$  & $0$ & $1/8$ & $25/72$ & $1/36$ \\ 
\end{tabular}
\end{table}

\begin{table}[h!]
\centering
\caption{$\method{3}{5}{2}$ - Fifth order three-derivative Runge-Kutta scheme using two stages \cite{TurTur2017}.}
\label{tab:ThDRK5-2s}
\begin{tabular}{c|cc|cc|cc}
$0$ & $0$   & $0$   & $0$   & $0$ & $0$   & $0$ \\
$2/5$ & $2/5$   & $0$   & $2/25$ & $0$  & $4/375$ & $0$ \\ \hline
  & $1$ & $0$ & $1/2$ & $0$ & $1/16$ & $5/48$
\end{tabular}
\end{table}

\begin{table}[h!]
\centering
\caption{$\method{3}{7}{3}$ - Seventh order three-derivative Runge-Kutta scheme using three stages \cite{TurTur2017}. The coefficients are given by $c_2 = \frac{3 - \sqrt{2}}{7}$, $c_3 = \frac{3 + \sqrt{2}}{7}$, $a_{32}^{(3)} = \frac{122 + 71\sqrt{2}}{7203}$, $b_1^{(3)} = \frac{1}{30}$, $b_2^{(3)} = \frac{1}{15} + \frac{13\sqrt{2}}{480}$, $b_3^{(3)} = \frac{1}{15} - \frac{13\sqrt{2}}{480}$. }
\label{tab:ThDRK7-3s}
\begin{tabular}{c|ccc|ccc|ccc}
$0$ & $0$   & $0$ & $0$   & $0$   & $0$ & $0$  & $0$   & $0$ & $0$  \\
$c_2$ & $c_2$ & $0$  & $0$   & $c_2^2/2$ & $0$  & $0$  & $c_2^3/6$ & $0$ & $0$  \\
$c_3$ & $c_3$   & $0$ & $0$   & $c_3^2/2$ & $0$ & $0$  & $c_2^3/6 - a_{32}^{(3)}$ & $a_{32}^{(3)}$ & $0$ \\[1pt] \hline
  & $1$ & $0$ & $0$  & $1/2$ & $0$ & $0$  & $b_1^{(3)}$ & $b_2^{(3)}$ & $b_3^{(3)}$
\end{tabular}
\end{table}

\begin{table}[h!]
\centering
\caption{$\method{4}{6}{2}$ - Sixth order four-derivative Runge-Kutta scheme using two stages.}
\label{tab:FDRK6-2s}
\begin{tabular}{c|cc|cc|cc|cc}
$0$ & $0$   & $0$   & $0$   & $0$ & $0$   & $0$ & $0$   & $0$\\
$1/3$ & $1/3$   & $0$   & $1/18$ & $0$  & $1/162$ & $0$ & $1/1944$   & $0$\\ \hline
  & $1$ & $0$ & $1/2$ & $0$ & $1/6$ & $0$  & $1/60$ & $1/40$
\end{tabular}
\end{table}

\FloatBarrier

%\bibliographystyle{spmpsci}      % mathematics and physical scien7ces
%\bibliographystyle{unsrt}            % in order of appearance (this is to make sure I have a good list of references -DS)
\bibliographystyle{abbrv}            % in order of appearance (this is to make sure I have a good list of references -DS)
\bibliography{ListPaper}            % name your BibTeX data base

\end{document}